\documentclass[12pt,british]{article}

\usepackage{babel}
\usepackage{amsmath,amsfonts,amssymb,amsthm}
\usepackage{enumerate,url}
\usepackage{eucal,mathtools}
\usepackage{tikz}
\usepackage{tikz-cd}
\usetikzlibrary{arrows,decorations.markings}

\tikzset{commutative diagrams/.cd,arrow style=tikz,diagrams={>=stealth'}}

\usepackage[bb=boondox]{mathalfa} 

\usepackage[noBBpl]{mathpazo}

\usepackage{mathabx}

\usepackage[matrix,arrow]{xy}

\usepackage[
bookmarks=true,
breaklinks=true,
bookmarksnumbered = true,
colorlinks= true,
urlcolor= green,
anchorcolor = yellow,
citecolor=blue,
]{hyperref}                   

\makeatletter
\renewcommand{\theenumi}{\alph{enumi}}
\renewcommand{\theenumii}{\roman{enumii}}

\renewcommand{\p@enumii}{}

\def\@enum@{\list{\csname label\@enumctr\endcsname}%
          {\usecounter{\@enumctr}\def\makelabel##1{
\normalfont\ignorespaces\emph{{##1}~}}
\setlength{\labelsep}{3pt}
\setlength{\parsep}{0pt}
\setlength{\itemsep}{0pt}
\setlength{\leftmargin}{0pt}
\setlength{\labelwidth}{0pt}
\setlength{\listparindent}{\parindent}
\setlength{\itemsep}{0pt}
\setlength{\itemindent}{0pt}
\topsep=3pt plus 1pt minus 1 pt}}

\makeatother

\renewcommand{\epsilon}{\ensuremath{\varepsilon}}
\renewcommand{\phi}{\ensuremath{\varphi}}
\newcommand{\vide}{\ensuremath{\varnothing}}
\renewcommand{\to}{\ensuremath{\longrightarrow}}
\renewcommand{\mapsto}{\ensuremath{\longmapsto}}

\newcommand{\N}{\ensuremath{\mathbb N}}
\newcommand{\Z}{\ensuremath{\mathbb Z}}
\newcommand{\dt}{\ensuremath{\mathbb D}^{2}}
\newcommand{\St}[1][2]{\ensuremath{\mathbb S}^{#1}}
\newcommand{\FF}{\ensuremath{\mathbb F}}
\newcommand{\F}[1][n]{\ensuremath{\FF_{{#1}}}}
\newcommand{\rp}{\ensuremath{\mathbb{R}P^2}}

\newcommand{\sn}[1][n]{\ensuremath{S_{{#1}}}}

\DeclareRobustCommand*{\up}[1]{\ensuremath{\textsuperscript{#1}}}
\renewcommand{\th}{\up{th}}
\newcommand{\ft}[1][n]{\ensuremath{\Delta_{#1}^{2}}}

\renewcommand{\ker}[1]{\ensuremath{\operatorname{\text{Ker}}\left({#1}\right)}}
\newcommand{\im}[1]{\ensuremath{\operatorname{\text{Im}}\left({#1}\right)}}
\newcommand{\aut}[1]{\ensuremath{\operatorname{\text{Aut}}\left({#1}\right)}}

\newcommand{\id}{\ensuremath{\operatorname{\text{Id}}}}
\newcommand{\hooklongrightarrow}{\lhook\joinrel\longrightarrow}

\newcommand{\quat}[1][8]{\ensuremath{\mathcal{Q}_{#1}}}

\newcommand{\cd}[1]{\ensuremath{\operatorname{\text{cd}}(#1)}}

\makeatletter
\def\@map#1#2[#3]{\mbox{$#1 \colon\thinspace #2 \to #3$}}
\def\map#1#2{\@ifnextchar [{\@map{#1}{#2}}{\@map{#1}{#2}[#2]}}
\makeatother

\newcommand{\brak}[1]{\ensuremath{\left\{ #1 \right\}}}
\newcommand{\ang}[1]{\ensuremath{\left\langle #1\right\rangle}}

\newcommand{\setangr}[2]{\ensuremath{\ang{#1 \,\left\lvert \, #2 \right.}}}
\newcommand{\setangl}[2]{\ensuremath{\ang{\left. #1 \,\right\rvert \, #2}}}

\newcommand{\lhra}{\lhook\joinrel\longrightarrow}

\newcommand{\setr}[2]{\ensuremath{\brak{#1 \,\left\lvert \, #2 \right.}}}
\newcommand{\setl}[2]{\ensuremath{\brak{\left. #1 \,\right\rvert \, #2}}}

\newtheoremstyle{theoremm}{}{}{\itshape}{}{\scshape}{.}{ }{}

\theoremstyle{theoremm}
\newtheorem{thm}{Theorem}

\newtheorem{prop}[thm]{Proposition}
\newtheorem{cor}[thm]{Corollary}

\newtheoremstyle{remarkk}{}{}{}{}{\scshape}{.}{ }{}

\theoremstyle{remarkk}

\newtheorem{rem}[thm]{Remark}
\newtheorem{rems}[thm]{Remarks}

\newtheoremstyle{comment}{}{}{\bfseries}{}{\bfseries}{:}{ }{}
\theoremstyle{comment}

\newcommand{\reth}[1]{Theorem~\protect\ref{th:#1}}

\newcommand{\repr}[1]{Proposition~\protect\ref{prop:#1}}
\newcommand{\reco}[1]{Corollary~\protect\ref{cor:#1}}
\newcommand{\resec}[1]{Section~\protect\ref{sec:#1}}
\newcommand{\rerem}[1]{Remark~\protect\ref{rem:#1}}
\newcommand{\rerems}[1]{Remarks~\protect\ref{rem:#1}}
\newcommand{\req}[1]{equation~(\protect\ref{eq:#1})}
\newcommand{\reqref}[1]{(\protect\ref{eq:#1})}

\newcommand{\rpminus}[1][n-2]{\ensuremath{P_{#1}(\rp\,\setminus\brak{x_{1},x_{2}})}}
\newcommand{\keromega}[2][n]{\ensuremath{\ker{\widehat{\iota}_{#2}\!\left\lvert_{\Omega_{#1}}\right.\!}}}

\begin{document}

\title{
The inclusion of configuration spaces of surfaces in Cartesian products, its induced homomorphism, and the virtual cohomological dimension of the braid groups of $\St$ and $\rp$}


\author{DACIBERG~LIMA~GON\c{C}ALVES\\
Departamento de Matem\'atica - IME-USP,\\
Caixa Postal~66281~-~Ag.~Cidade de S\~ao Paulo,\\ 
CEP:~05314-970 - S\~ao Paulo - SP - Brazil.\\
e-mail:~\url{dlgoncal@ime.usp.br}\vspace*{4mm}\\
JOHN~GUASCHI\\
Normandie Universit\'e, UNICAEN,\\
Laboratoire de Math\'ematiques Nicolas Oresme UMR CNRS~\textup{6139},\\
CS 14032, 14032 Caen Cedex 5, France.\\
e-mail:~\url{john.guaschi@unicaen.fr}}

\date{\today}

\begingroup
\renewcommand{\thefootnote}{}
\endgroup 

\maketitle

\begin{abstract}
\noindent \emph{Let $S$ be a surface, perhaps with boundary, and either compact, or with a finite number of points removed from the interior of the surface. We consider the inclusion $\map{\iota}{F_{n}(S)}[\prod_{1}^{n}\, S]$ of the $n\up{th}$ configuration space $F_{n}(S)$ of $S$ into the $n$-fold Cartesian product of $S$, as well as the induced homomorphism $\map{\iota_{\#}}{P_{n}(S)}[\prod_{1}^{n}\, \pi_{1}(S)]$, where $P_{n}(S)$ is the $n$-string pure braid group of $S$. Both $\iota$ and $\iota_{\#}$ were studied initially by J.~Birman who conjectured that $\ker{\iota_{\#}}$ is equal to the normal closure of the Artin pure braid group $P_{n}$ in $P_{n}(S)$. The conjecture was later proved by C.~Goldberg for compact surfaces without boundary different from the $2$-sphere $\St$ and the projective plane $\rp$. In this paper, we prove the conjecture for $\St$ and $\rp$. In the case of $\rp$, we prove that $\ker{\iota_{\#}}$ is equal to the commutator subgroup of $P_{n}(\rp)$, we show that it may be decomposed in a manner similar to that of $P_{n}(\St)$ as a direct sum of a torsion-free subgroup $L_{n}$ and the finite cyclic group generated by the full twist braid, and we prove that $L_{n}$ may be written as an iterated semi-direct product of free groups. Finally, we show that the groups $B_n(\St)$ and $P_n(\St)$ (resp.\ $B_n(\rp)$ and $P_n(\rp)$) have finite virtual cohomological dimension equal to $n-3$ (resp.\ $n-2$), where $B_{n}(S)$ denotes the full $n$-string braid group of $S$. This allows us to determine the virtual cohomological dimension of the mapping class groups of the mapping class groups of $\St$ and $\rp$ with marked points, which in the case of $\St$, reproves a result due to J.~Harer.}
\end{abstract}

\section{Introduction}\label{sec:intro}

Let $S$ be a connected surface, perhaps with boundary, and either compact, or with a finite number of points removed from the interior of the surface. 
The \emph{$n\up{th}$ configuration space} of $S$ is defined by:
\begin{equation*}
F_n(S)=\setr{(x_{1},\ldots,x_{n})\in S^{n}}{\text{$x_{i}\neq x_{j}$ if $i\neq j$}}.
\end{equation*}
It is well known that $\pi_1(F_n(S))\cong P_n(S)$, the \emph{pure braid group} of $S$ on $n$ strings, and that $\pi_1(F_n(S)/\sn)\cong B_n(S)$, the \emph{braid group} of $S$ on $n$ strings, where $F_n(S)/\sn$ is the quotient space of $F_{n}(S)$ by the free action of the symmetric group $\sn$ given by permuting coordinates~\cite{FaN,FoN}. If $S$ is the $2$-disc $\dt$ then $B_{n}(\dt)$ (resp.\ $P_{n}(\dt)$) is the Artin braid group $B_{n}$ (resp.\ the Artin pure braid group $P_{n}$). The canonical projection $F_n(S)\to F_n(S)/\sn$ is a regular $n!$-fold covering map, and thus gives rise to the following short exact sequence:
\begin{equation}\label{eq:sessym}
1\to P_n(S) \to B_n(S) \to \sn\to 1.
\end{equation}
If $\dt$ is a topological disc lying in the interior of $S$ and that contains the basepoints of the braids then the inclusion $\map{j}{\dt}[S]$ induces a group homomorphism $\map{j_{\#}}{B_{n}}[B_{n}(S)]$. This homomorphism is injective if $S$ is different from the $2$-sphere $\St$ and the real projective plane $\rp$~\cite{Bi,Gol}. Let $\map{j_{\#}\left\lvert_{P_{n}}\right.}{P_{n}}[P_{n}(S)]$ denote the restriction of $j_{\#}$ to the corresponding pure braid groups. If $\beta\in B_{n}$ then we shall denote its image $j_{\#}(\beta)$ in $B_{n}(S)$ simply by $\beta$. It is well known that the centre of $B_{n}$ and of $P_{n}$ is infinite cyclic, generated by the full twist braid that we denote by $\ft$, and that $\ft$, considered as an element of $B_{n}(\St)$ or of $B_{n}(\rp)$, is of order $2$ and generates the centre. If $G$ is a group then we denote its commutator subgroup by $\Gamma_{2}(G)$, its Abelianisation by $G\up{Ab}$, and if $H$ is a subgroup of $G$ then we denote its normal closure in $G$ by $\ang{\!\ang{H}\!}_{G}$.

Let $\prod_1^n\, S=S\times \cdots \times S$ denote the $n$-fold Cartesian product of $S$ with itself, let $\map{\iota_{n}}{F_n(S)}[\prod_1^n\, S]$ be the inclusion map, and let $\map{\iota_{n\#}}{\pi_1(F_n(S))}[\pi_1\left( \prod_1^n\, S\right)]$ denote the induced homomorphism on the level of fundamental groups. To simplify the notation, we shall often just write $\iota$ and $\iota_{\#}$ if $n$ is given. The study of $\iota_{\#}$ was initiated by Birman in 1969~\cite{Bi}. She had conjectured that $\ang{\!\!\ang{\,\im{j_{\#}\left\lvert_{P_{n}}\right.\!}}\!\!}_{P_{n}(S)}= \ker{\iota_{\#}}$ if $S$ is a compact orientable surface, but states without proof that her conjecture is false if $S$ is of genus greater than or equal to $1$~\cite[page~45]{Bi}. However, Goldberg proved the conjecture several years later in both the orientable and non-orientable cases for compact surfaces without boundary different from $\St$ and $\rp$~\cite[Theorem~1]{Gol}. In connection with the study of Vassiliev invariants of surface braid groups, González-Meneses and Paris showed that $\ker{\iota_{\#}}$ is also normal in $B_n(S)$, and that the resulting quotient is isomorphic to the semi-direct product $\pi_1\left(\prod_1^n\, S\right)\rtimes \sn$, where the action is given by permuting coordinates (their work was within the framework of compact, orientable surfaces without boundary, but their construction is valid for any surface $S$)~\cite{GMP}. In the case of $\rp$, this result was reproved using geometric methods~\cite{T}.


If $S=\St$, $\ker{\iota_{\#}}$ is clearly equal to $P_n(\St)$, and so by~\cite[Theorem~4]{GG2}, it may be decomposed as: 
\begin{equation}\label{eq:pns2sum}
\ker{\iota_{\#}}=P_n(\St)\cong P_{n-3}(\St\setminus \brak{x_1,x_2,x_3})\times \Z_2,
\end{equation}
where the first factor of the direct product is torsion free, and the $\Z_{2}$-factor is generated by $\ft$. 

The aim of this paper is to resolve Birman's conjecture for surfaces without boundary in the remaining cases, namely $S=\St$ or $\rp$, to determine the cohomological dimension of $B_{n}(S)$ and $P_{n}(S)$, where $S$ is one of these two surfaces, and to elucidate the structure of $\ker{\iota_{\#}}$ in the case of $\rp$.
In \resec{kerrp2}, we start by considering the case $S=\rp$, we study $\ker{\iota_{\#}}$, which we denote by $K_{n}$, and we show that it admits a decomposition similar to that of \req{pns2sum}.


\begin{prop}\label{prop:prop3}
Let $n\in \N$.
\begin{enumerate}[(a)]
\item\label{it:prop3a} 
\begin{enumerate}[(i)]
\item\label{it:prop3ai} Up to isomorphism, the homomorphism $\map{\iota_{\#}}{\pi_1( F_n(\rp))} [\pi_1(\Pi_{1}^{n}( \rp))]$ coincides with Abelianisation. In particular, $K_{n}=\Gamma_{2}(P_{n}(\rp))$.
\item\label{it:prop3aii} If $n\geq 2$ then there exists a torsion-free subgroup $L_{n}$ of $K_{n}$ such that $K_{n}$ is isomorphic to the direct sum of $L_{n}$ and the subgroup $\ang{\ft}$ generated by the full twist that is isomorphic to $\Z_{2}$.
\end{enumerate}
\item\label{it:prop3b} If $n\geq 2$ then any subgroup of $P_n(\rp)$ that is normal in $B_n(\rp)$ and that properly contains $K_{n}$ possesses an element of order $4$.
\end{enumerate}
\end{prop}

Note that if $n=1$ then $B_{1}(\rp)=P_{1}(\rp)\cong \Z_{2}$ and $\ft[1]$ is the trivial element, so parts~(\ref{it:prop3a})(\ref{it:prop3aii}) and~(\ref{it:prop3b}) do not hold. Part~(\ref{it:prop3a})(\ref{it:prop3ai}) will be proved in \repr{abeliota}. We shall see later on in \rerem{Lnuniqueness} that there are precisely $2^{n(n-2)}$ subgroups that satisfy the conclusions of part~(\ref{it:prop3a})(\ref{it:prop3aii}), and to prove the statement, we shall exhibit an explicit torsion-free subgroup $L_{n}$. We then prove Birman's conjecture for $\St$ and $\rp$, using \repr{prop3}(\ref{it:prop3a})(\ref{it:prop3ai}) in the case of $\rp$.

\begin{thm}\label{th:birman}
Let $S$ be one of $\St$ or $\rp$, and let $n\geq 1$. Then $\ang{\!\ang{\,\im{j_{\#}\left\lvert_{P_{n}}\right.\!}}\!}_{P_{n}(S)}= \ker{\iota_{\#}}$.
\end{thm}

In \resec{propsLn}, we analyse $L_{n}$ in more detail, and we show that it may be decomposed as an iterated semi-direct product of free groups. 

\begin{thm}\label{th:th4}
Let $n\geq 3$. Consider the Fadell-Neuwirth short exact sequence:
\begin{equation}\label{eq:fnpnp2}
1\to \rpminus \to P_{n}(\rp) \stackrel{q_{2\#}}{\to} P_{2}(\rp) \to 1,
\end{equation}
where $q_{2\#}$ is given geometrically by forgetting the last $n-2$ strings. Then $L_{n}$ may be identified with the kernel of the composition
\begin{equation*}
\rpminus \to P_n(\rp) \stackrel{\iota_{\#}}{\to} \underbrace{\Z_2 \times\cdots\times \Z_2}_{\text{$n$ copies}},
\end{equation*}
where the first homomorphism is that appearing in \req{fnpnp2}. The image of this composition is the product of the last $n-2$ copies of $\Z_2$. In particular, $L_{n}$ is of index $2^{n-2}$ in $\rpminus$. Further, $L_{n}$ is isomorphic to an iterated semi-direct product of free groups of the form $\F[2n-3]\rtimes (\F[2n-5]\rtimes (\cdots \rtimes(\F[5]\rtimes \F[3])\cdots))$, where for all $m\in \N$, $\F[m]$ denotes the free group of rank $m$.
\end{thm}

In the semi-direct product decomposition of $L_{n}$, note that every factor acts on each of the preceding factors. This is also the case for $\rpminus$ (see \req{semifn}), and as we shall see in \rerems{artin}(\ref{it:artin1}), this implies an Artin combing-type result for this group. Analysing these semi-direct products in more detail, we obtain the following results. 

\begin{prop}\label{prop:abelianL}
If $n\geq 3$ then:
\begin{enumerate}[(a)]
\item\label{it:abparta} $\bigl(\rpminus\bigr)\up{Ab}\cong \Z^{2(n-2)}$.
\item\label{it:abpartb} $(L_{n})\up{Ab}\cong \Z^{n(n-2)}$.
\end{enumerate}
\end{prop}

In two papers in preparation, we shall analyse the homotopy fibre of $\iota$, as well as the induced homomorphism $\iota_{\#}$ when $S=\St$ or $\rp$~\cite{GG10}, and when $S$ is a space form manifold of dimension different from two~\cite{GGG}. In the first of these papers, we shall also see that $L_{n}$ is closely related to the fundamental group of an orbit configuration space of the open cylinder.

In \resec{vcd}, we study the virtual cohomological dimension of the braid groups of $\St$ and $\rp$. Recall from~\cite[page~226]{Br} that if a group $\Gamma$ is virtually torsion-free then all finite index torsion-free subgroups of $\Gamma$ have the same cohomological dimension by Serre's theorem, and this dimension is defined to be the \emph{virtual cohomological dimension} of $\Gamma$. Using equations~\reqref{pns2sum} and~\reqref{fnpnp2}, we prove the following result, namely that if $S=\St$ or $\rp$, the groups $B_n(S)$ and $P_n(S)$ have finite virtual cohomological dimension, and we compute these dimensions.

\begin{thm}\label{th:prop12}\mbox{}
\begin{enumerate}[(a)]
\item\label{it:harer1} Let $n\geq 4$. Then the virtual cohomological dimension of both $B_n(\St)$ and $P_n(\St)$ is equal to the cohomological dimension of the group $P_{n-3}(\St\setminus \brak{x_1,x_2,x_3})$. Furthermore, for all $m\geq 1$, the cohomological dimension of the group $P_{m}(\St\setminus\brak{x_1,x_2,x_3})$ is equal to $m$.
\item Let $n\geq 3$. Then the virtual cohomological dimension of both $B_n(\rp)$ and $P_n(\rp)$ is equal to the cohomological dimension of the group $P_{n-2}(\rp\setminus\brak{x_1, x_2})$. Furthermore, for all $m\geq 1$, the cohomological dimension of the group $P_{m}(\rp\setminus \brak{x_1, x_2})$ is equal to $m$.
\end{enumerate}
\end{thm}

The methods of the proof of \reth{prop12} have recently been applied to compute the cohomological dimension of the braid groups of all other compact surfaces (orientable and non orientable) without boundary~\cite{GGM}. \reth{prop12} also allows us to deduce the virtual cohomological dimension of the punctured mapping class groups of $\St$ and $\rp$. If $n\geq 0$, let $\mathcal{MCG}(S,n)$ denote the mapping class group of a connected, compact surface $S$ relative to an $n$-point set. If $S$ is orientable then Harer determined the virtual cohomological dimension of $\mathcal{MCG}(S,n)$~\cite[Theorem 4.1]{H}. In the case of $\St$ and $\dt$, he obtained the following results:
\begin{enumerate}[(a)]
\item if $n\geq 3$, the virtual cohomological dimension of $\mathcal{MCG}(\St,n)$ is equal to $n-3$.
\item if $n\geq 2$, the cohomological dimension of $\mathcal{MCG}(\dt,n)$ is equal to $n-1$ (recall that $\mathcal{MCG}(\dt,n)$ is isomorphic to $B_{n}$~\cite{Bi2}).
\end{enumerate}
As a consequence of \reth{prop12}, we are able to compute the virtual cohomological dimension of $\mathcal{MCG}(S,n)$ for $S=\St$ and $\rp$.

\begin{cor}\label{cor:vcdmcg}
Let $n\geq 4$ (resp.\ $n\geq 3$). Then the virtual cohomological dimension of $\mathcal{MCG}(\St,n)$ (resp.\ $\mathcal{MCG}(\rp,n)$) is finite, and is equal to $n-3$ (resp.\ $n-2$).
\end{cor}

If $S=\St$ or $\rp$ then for the values of $n$ given by \reth{prop12} and \reco{vcdmcg}, the virtual cohomological dimension of $\mathcal{MCG}(S,n)$ is equal to that of $B_{n}(S)$. If $S=\St$, we thus recover the corresponding result of Harer.

\subsection*{Acknowledgements}

This work took place during the visits of the first author to the Laboratoire de Math\'e\-matiques Nicolas Oresme during the periods 2\up{nd}--23\up{rd}~December 2012, 29\up{th}~November--22\up{nd}~December 2013 and 4\up{th}~October--1\up{st}~November~2014, and of the visits of the second author to the Departamento de Matem\'atica do IME~--~Universidade de S\~ao Paulo during the periods 10\up{th}~November--1\up{st}~December 2012,~1\up{st}--21\up{st} July 2013 and 10\up{th}~July--2\up{nd}~August 2014, and was supported by the international Cooperation Capes-Cofecub project n\up{o}~Ma~733-12 (France) and n\up{o}~1716/2012 (Brazil), and the CNRS/Fapesp programme n\up{o}~226555 (France) and n\up{o}~2014/50131-7 (Brazil).

\section{The structure of $K_{n}$, and Birman's conjecture for $\St$ and $\rp$}\label{sec:kerrp2}

Let $n\in \N$. As we mentioned in the introduction, if $S$ is a surface different from $\St$ and $\rp$, the kernel of the homomorphism $\map{\iota_{\#}}{P_n(S)} [{\pi_1\left(\prod_1^n\, S\right)}]$ was studied in~\cite{Bi,Gol}, and that if $S=\St$ then $\ker{\iota_{\#}}=P_n(\St)$. In the first part of this section, we recall a presentation of $P_{n}(\rp)$, and we prove \repr{prop3}(\ref{it:prop3a})(\ref{it:prop3ai}). The second part of this section is devoted to proving the rest of \repr{prop3} and \reth{birman}, the latter being Birman's conjecture for $\St$ and $\rp$. 

Consider the model of $\rp$ given by identifying antipodal boundary points of $\dt$. We equip $F_{n}(\rp)$ with a basepoint $(x_{1},\ldots,x_{n})$. For $1\leq i<j\leq n$ (resp.\ $1\leq k\leq n$), we define the element $A_{i,j}$ (resp.\ $\tau_{k}$, $\rho_{k}$) of $P_{n}(\rp)$ by the geometric braids depicted in Figure~\ref{fig:gens}.
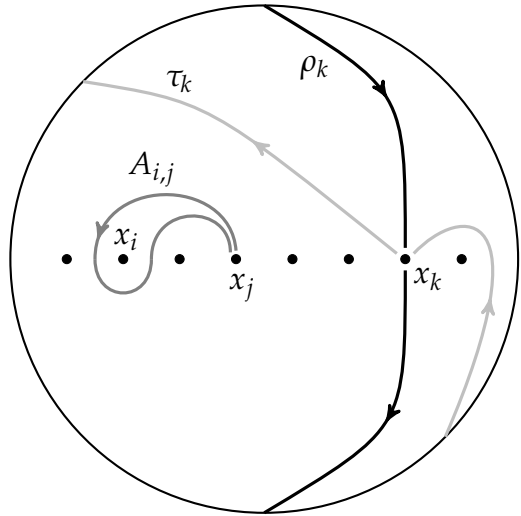
\begin{figure}[h]
\hfill
\begin{tikzpicture}[scale=0.75]
\draw[thick] (0,0) circle(4.5);

\foreach \k in {-3.5,-2.5,...,3.5}
{\draw[fill] (\k,0) circle [radius=0.08];};

\node at (2.9, -0.35) {$x_{k}$};
\node at (-0.4, -0.5) {$x_{j}$};
\node at (-2.45, 0.35) {$x_{i}$};

\node at (-2, 1.6) {$A_{i,j}$};
\node at (-1.5, 3.15) {$\tau_{k}$};
\node at (0.9, 3.4) {$\rho_{k}$};

\draw[very thick, decoration={markings,mark=at position 0.5 with {\arrow{stealth'}}},postaction={decorate}] (2.5,-0.2) .. controls (2.5,-3) .. (0,-4.5);

\draw[very thick, decoration={markings,mark=at position 0.5 with {\arrow{stealth'}}},postaction={decorate}] (0,4.5) .. controls (2.5,3) .. (2.5,0.2);

\draw[very thick, color=gray!50,decoration={markings,mark=at position 0.5 with {\arrow{stealth'}}},postaction={decorate}] (2.35,0.1)  .. controls (-1,2.8) .. (-3.21,3.15);

\draw[very thick, color=gray!50,decoration={markings,mark=at position 0.5 with {\arrow{stealth'}}},postaction={decorate}] (3.21,-3.15) .. controls (5,1) and (3.5,1) .. (2.65,0.1);

%

\draw[very thick, color=gray, decoration={markings,mark=at position 0.9 with {\arrow{stealth'}}},postaction={decorate}] (-0.5,0.15) .. controls (-0.6,1.5) and (-3,1.5) .. (-3,0);

\draw[very thick, color=gray] (-3,0)  .. controls (-3,-0.8) and (-2,-0.8)  .. (-2,0);

\draw[very thick, color=gray] (-2,0)  .. controls (-2,1) and (-0.6,1)  .. (-0.6,0.12);
\end{tikzpicture}
\hspace*{\fill}
\caption{The elements $A_{i,j}$, $\tau_{k}$ and $\rho_{k}$ of $P_{n}(\rp)$.}
\label{fig:gens}
\end{figure}
Note that the arcs represent the projections of the strings onto $\rp$, so that all of the strings of the given braid are vertical, with the exception of the $j\up{th}$ (resp.\ $k\up{th}$) string that is based at the point $x_{j}$ (resp.\ $x_{k}$).


\begin{thm}[{\cite[Theorem~4]{GG4}}]\label{th:basicpres}
Let $n\in\N$. The following constitutes a presentation of the pure braid group $P_n(\rp)$:
\begin{enumerate}
\item[\underline{\textbf{generators:}}] $A_{i,j}$, $1\leq i<j\leq n$, and $\tau_k$, $1\leq k\leq n$.
\item[\underline{\textbf{relations:}}]\mbox{}
\begin{enumerate}[(a)]
\item\label{it:rel1} the Artin relations between the $A_{i,j}$ emanating from those of $P_n$:
\begin{equation}\label{eq:artinaij}
A_{r,s}A_{i,j}A_{r,s}^{-1}\!=\!
\begin{cases}
\! A_{i,j} & \text{if $i<r<s<j$ or $r<s<i<j$}\\
\! A_{i,j}^{-1} A_{r,j}^{-1}  A_{i,j} A_{r,j} A_{i,j} & \text{if $r<i=s<j$}\\
\! A_{s,j}^{-1} A_{i,j} A_{s,j} & \text{if $i=r<s<j$}\\
\! A_{s,j}^{-1}A_{r,j}^{-1} A_{s,j} A_{r,j} A_{i,j} A_{r,j}^{-1} A_{s,j}^{-1} A_{r,j} A_{s,j} & \text{if $r<i<s<j$.}
\end{cases}
\end{equation}
\item\label{it:rel2} for all $1\leq i<j\leq n$, $\tau_i\tau_j\tau_i^{-1} = \tau_j^{-1} A_{i,j}^{-1}  \tau_j^2$.
\item\label{it:rel3} for all $1\leq i\leq n$, $\tau_i^2= A_{1,i}\cdots A_{i-1,i} A_{i,i+1} \cdots A_{i,n}$.
\item\label{it:rel4} for all $1\leq i<j\leq n$ and $1\leq k\leq n$, $k\neq j$,
\begin{equation*}\label{eq:relspn}
\tau_k A_{i,j}\tau_k^{-1}=
\begin{cases}
A_{i,j} & \text{if $j<k$ or $k<i$}\\
\tau_j^{-1} A_{i,j}^{-1} \tau_j & \text{if $k=i$}\\
\tau_j^{-1} A_{k,j}^{-1} \tau_j A_{k,j}^{-1} A_{i,j} A_{k,j} \tau_j^{-1} A_{k,j} \tau_j & \text{if $i<k<j$}.
\end{cases}
\end{equation*}
\end{enumerate}
\end{enumerate}
\end{thm}

This enables us to prove that $\iota_{\#}$ is in fact Abelianisation, which is part~(\ref{it:prop3a})(\ref{it:prop3ai}) of \repr{prop3}.
\begin{prop}\label{prop:abeliota}
Let $n\in \N$. The homomorphism $\map{\iota_{\#}}{P_n(\rp)} [{\pi_1(\prod_1^n\, \rp)}]$ is defined on the generators of \reth{basicpres} by $\iota_{\#}(A_{i,j})=(\overline{0},\ldots,\overline{0})$ for all $1\leq i<j\leq n$, and $\iota_{\#}(\tau_k)=(\overline{0},\ldots, \overline{0}, \underbrace{\overline{1}}_{\mathclap{\text{$k\up{th}$ position}}}, \overline{0}, \ldots,\overline{0})$ for all $1\leq k\leq n$. Further, $\iota_{\#}$ is Abelianisation, and $\ker{\iota_{\#}}=K_{n}= \Gamma_{2}(P_{n}(\rp))$.
\end{prop}

\begin{proof}
For $1\leq k\leq n$, let $\map{p_{k}}{F_{n}(\rp)}[\rp]$ denote projection onto the $k\up{th}$ coordinate. Observe that $\iota_{\#}=p_{1\#}\times \cdots \times p_{n\#}$, where $\map{p_{k\#}}{P_{n}(\rp)}[\pi_{1}(\rp)]$ is the induced homomorphism on the level of fundamental groups. Identifying $\pi_{1}(\rp)$ with $\Z_{2}$ and using the geometric realisation of Figure~\ref{fig:gens} of the generators of the presentation of $P_{n}(\rp)$ given by \reth{basicpres}, it is straightforward to check that for all $1\leq k,l\leq n$ and $1\leq i<j\leq n$, $p_{k\#}(A_{i,j})=\overline{0}$, $p_{k\#}(\tau_{l})=\overline{0}$ if $l\neq k$ and $p_{k\#}(\tau_{k})=\overline{1}$, and this yields the first part of the proposition. The second part follows easily from the presentation of the Abelianisation $(P_{n}(\rp))\up{Ab}$ of $P_{n}(\rp)$ obtained from \reth{basicpres}. More precisely, if we denote the Abelianisation of an element $x\in P_{n}(\rp)$ by $\overline{x}$, relations~(\ref{it:rel2}) and~(\ref{it:rel3}) imply respectively that for all $1\leq i<j\leq n$ and $1\leq k\leq n$, $\overline{A_{i,j}}$ and $\overline{\tau_{k}}^{2}$ represent the trivial element of $(P_{n}(\rp))\up{Ab}$. Since the remaining relations give no other information under Abelianisation, it follows that $(P_{n}(\rp))\up{Ab}\cong \Z_{2}\oplus \cdots \oplus \Z_{2}$, where $\overline{\tau_{k}}= (\overline{0},\ldots, \overline{0}, \underbrace{\overline{1}}_{\mathclap{\text{$k\up{th}$ position}}}, \overline{0}, \ldots,\overline{0})$ and $\overline{A_{i,j}}=(\overline{0},\ldots,\overline{0})$ via this isomorphism, and the Abelianisation homomorphism indeed coincides with $\iota_{\#}$ on $P_{n}(\rp)$.
\end{proof}

\begin{rems}\mbox{}\label{rem:propsK}
\begin{enumerate}[(a)]
\item\label{it:propsKa} Since $K_{n}=\Gamma_{2}(P_{n}(\rp))$, it follows immediately that $K_{n}$ is normal in $B_n(\rp)$, since $\Gamma_{2}(P_{n}(\rp))$ is characteristic in $P_n(\rp)$, and $P_n(\rp)$ is normal in $B_n(\rp)$.
\item A presentation of $K_{n}$ may be obtained by a long but routine computation using the Reidemeister-Schreier method, although it is not clear how to simplify the presentation. In \reth{th4}, we will provide an alternative description of $K_{n}$ using algebraic methods. 
\item\label{it:propsKc} In what follows, we shall use Van Buskirk's presentation of $B_{n}(\rp)$~\cite[page~83]{vB} whose generating set consists of the standard braid generators $\sigma_{1},\ldots,\sigma_{n-1}$ emanating from the $2$-disc, as well as the surface generators $\rho_{1},\ldots,\rho_{n}$ depicted in Figure~\ref{fig:gens}. We have the following relation between the elements $\tau_{k}$ and $\rho_{k}$:
\begin{equation*}
\tau_{k}=\rho_{k}^{-1} A_{k,k+1}\cdots A_{k,n} \quad \text{for all $1\leq k\leq n$,}
\end{equation*}
where for $1\leq i<j\leq n$, $A_{i,j}=\sigma_{j-1}\cdots \sigma_{i+1}\sigma_{i}^2\sigma_{i+1}^{-1}\cdots \sigma_{j-1}^{-1}$.
In particular, it follows from \repr{abeliota} that:
\begin{equation}\label{eq:rhotauk}
\text{$\iota_{\#}(\rho_{k})= \iota_{\#}(\tau_{k})=(\overline{0},\ldots, \overline{0}, \underbrace{\overline{1}}_{\mathclap{\text{$k\up{th}$ position}}}, \overline{0}, \ldots,\overline{0})$ for all $1\leq k\leq n$.}
\end{equation}
\end{enumerate}
\end{rems}

If $n\geq 2$, the full twist braid $\ft$, which may be defined by $\ft=(\sigma_{1}\cdots \sigma_{n-1})^{n}$, is of order~$2$~\cite[page~95]{vB}, it generates the centre of $B_{n}(\rp)$~\cite[Proposition~6.1]{Mu}, and is the unique element of $B_{n}(\rp)$ of order~$2$~\cite[Proposition~23]{GG3}. Since $\ft\in P_{n}(\rp)$, it thus belongs to the centre of $P_{n}(\rp)$, and just as for the Artin braid groups and the braid groups of $\St$, it generates the centre of $P_{n}(\rp)$:

\begin{prop}\label{prop:centrerp2}
Let $n\geq 2$. Then the centre $Z(P_{n}(\rp))$ of $P_{n}(\rp)$ is generated by $\ft$.
\end{prop}

\begin{proof}
We prove the result by induction on $n$. If $n=2$ then $P_2(\rp)\cong \quat$~\cite[page~87]{vB}, the quaternion group of order $8$, and the result follows since $\ft[2]$ is the element of $P_2(\rp)$ of order $2$. So suppose that $n\geq 3$. From the preceding remarks, $\ang{\ft}\subset Z(P_{n}(\rp))$. Conversely, let $x\in Z(P_{n}(\rp))$, and consider the following Fadell-Neuwirth short exact sequence:
\begin{equation*}
1\to \pi_{1}(\rp\setminus \brak{x_{1},\ldots,x_{n-1}})\to P_{n}(\rp)\xrightarrow{q_{(n-1)\#}} P_{n-1}(\rp)\to 1,
\end{equation*}
where $q_{(n-1)\#}$ is the surjective homomorphism induced on the level of fundamental groups by the projection $\map{q_{n-1}}{F_{n}(\rp)}[F_{n-1}(\rp)]$ onto the first $n-1$ coordinates. Now $q_{(n-1)\#}(x)\in Z(P_{n-1}(\rp))$ by surjectivity, and thus $q_{(n-1)\#}(x)=\Delta_{n-1}^{2\epsilon}$ for some $\epsilon\in \brak{0,1}$ by the induction hypothesis. Further, $q_{(n-1)\#}(\ft)=\ft[n-1]$, hence
\begin{equation*}
\Delta_{n}^{-2\epsilon} x\in \operatorname{Ker}(q_{(n-1)\#}) \cap Z(P_{n}(\rp)),
\end{equation*}
and thus $\Delta_{n}^{-2\epsilon} x\in Z(\operatorname{Ker}(q_{(n-1)\#}))$. But $Z(\operatorname{Ker}(q_{(n-1)\#}))$ is trivial because $\operatorname{Ker}(q_{(n-1)\#})$ is a free group of rank $n-1$. This implies that $x\in \ang{\ft}$ as required.
\end{proof}

We now prove \repr{prop3}.

\pagebreak

\begin{proof}[Proof of \repr{prop3}]
Let $n\geq 3$.
\begin{enumerate}[(a)]
\item Recall that part~(\ref{it:prop3a})(\ref{it:prop3ai}) of \repr{prop3} was proved in \repr{abeliota}, so let us prove part~(\ref{it:prop3aii}). The projection $\map{q_2}{F_n(\rp)}[F_{2}(\rp)]$ onto the first two coordinates gives rise to the Fadell-Neuwirth short exact sequence~\reqref{fnpnp2}. Since $K_{n}=\Gamma_{2}(P_{n}(\rp))$ by \repr{abeliota}, the image of the restriction $q_{2\#}\!\left\vert_{K_{n}}\right.$ of $q_{2\#}$ to $K_{n}$ is the subgroup $\Gamma_{2}(P_{2}(\rp))=\ang{\ft[2]}$, and so we obtain the following commutative diagram:
\begin{equation}\label{eq:fnpnp2ext}
\begin{tikzcd}
1\arrow{r} & K_{n}\cap \rpminus \arrow{r}\arrow{d} & K_{n}  \arrow{r}{q_{2\#}\left\vert_{K_{n}}\right.}\arrow{d} & \ang{\ft[2]} \arrow{r}\arrow{d} & 1\\
1\arrow{r} & \rpminus \arrow{r} & P_{n}(\rp) \arrow{r}{q_{2\#}} & P_{2}(\rp)  \arrow{r} & 1,
\end{tikzcd}
\end{equation}
where the vertical arrows are inclusions. Now $\ang{\ft[2]}\cong \Z_{2}$,  so $K_{n}$ is an extension of the group $\ker{q_{2\#}\!\left\vert_{K_{n}}\right.\!}=K_{n}\cap \rpminus$  by $\Z_2$. The fact that $q_{2\#}(\ft)=\ft[2]$ implies that the upper short exact sequence splits, a section being defined by the correspondence $\ft[2] \mapsto \ft$, and since $\ft\in Z(P_n(\rp))$, the action by conjugation on $\ker{q_{2\#}\!\left\vert_{K_{n}}\right.\!}$ is trivial. Part~(\ref{it:prop3a}) of the proposition follows by taking $L_{n}=\ker{q_{2\#}\!\left\vert_{K_{n}}\right.\!}$ and by noting that $\rpminus$ is torsion free. 

\item Recall first that any torsion element in $P_n(\rp)\setminus \ang{\ft}$ is of order $4$~\cite[Corollary~19 and Proposition~23]{GG3}, and is conjugate in $B_n(\rp)$ to one of $a^n$ or $b^{n-1}$, where $a=\rho_{n}\sigma_{n-1}\cdots \sigma_{1}$ and $b= \rho_{n-1}\sigma_{n-2}\cdots \sigma_{1}$ satisfy:
\begin{equation}\label{eq:anbn}
\text{$a^{n}=\rho_{n}\cdots \rho_{1}$ and $b^{n-1}=\rho_{n-1}\cdots \rho_{1}$}
\end{equation}
by~\cite[Proposition~10]{GG9}. Let $N$ be a normal subgroup of $B_{n}(\rp)$ that satisfies $K_{n} \subsetneqq N \subset P_{n}(\rp)$. We claim that for all $u\in {\pi_1(\prod_1^n\, \rp)}$ (which we identify henceforth with $\Z_{2}\oplus \cdots \oplus \Z_{2}$), exactly one of the following two conditions holds:
\begin{enumerate}[(i)]
\item $N \cap \iota_{\#}^{-1}(\brak{u})$ is empty.
\item\label{it:nint} $\iota_{\#}^{-1}(\brak{u})$ is contained in $N$.
\end{enumerate}
To prove the claim, suppose that $x\in N \cap \iota_{\#}^{-1}(\brak{u})\neq \vide$, and let $y\in \iota_{\#}^{-1}(\brak{u})$. Now $\iota_{\#}(x)=\iota_{\#}(y)=u$, so there exists $k\in K_{n}$ such that $x^{-1}y=k$. Since $K_{n}\subset N$, it follows that $y=xk\in N$, which proves the claim. Further, $\iota_{\#}(a^{n})=(\overline{1},\ldots,\overline{1})$ and $\iota_{\#}(b^{n-1})=(\overline{1},\ldots,\overline{1},\overline{0})$ by \repr{abeliota} and equations~\reqref{rhotauk} and \reqref{anbn}, so by the claim it suffices to prove that there exists $z\in N$ such that $\iota_{\#}(z)\in \brak{(\overline{1},\ldots,\overline{1}), (\overline{1},\ldots,\overline{1}, \overline{0})}$, for then we are in case~(\ref{it:nint}) above, and it follows that one of $a^{n}$ and $b^{n-1}$ belongs to $N$.

It thus remains to prove the existence of such a $z$. Let $x\in N\setminus K_{n}$. Then $\iota_{\#}(x)$ contains an entry equal to $\overline{1}$ because  $K_{n}=\ker{\iota_{\#}}$. If $\iota_{\#}(x)=(\overline{1},\ldots,\overline{1})$ then we are done. So assume that $\iota_{\#}(x)$ also contains an entry that is equal to $\overline{0}$. 
By \req{rhotauk}, there exist $1\leq r<n$ and $1\leq i_{1}<\cdots < i_{r}\leq n$ such that $\iota_{\#}(\rho_{i_{1}}\cdots \rho_{i_{r}})=\iota_{\#}(x)$. It follows from the claim and the fact that $x\in N$ that $\rho_{i_{1}}\cdots \rho_{i_{r}}\in N$ also, and so without loss of generality, we may suppose that $x=\rho_{i_{1}}\cdots \rho_{i_{r}}$. Further, since $\iota_{\#}(x)$ contains both a $\overline{0}$ and a $\overline{1}$, there exists $1\leq j\leq r$ such that $p_{i_{j}\#}(x)=\overline{1}$ and $p_{(i_{j}+1)\#}(x)=\overline{0}$, the homomorphisms $p_{k\#}$ being those defined in the proof of \repr{abeliota}. Note that we consider the indices modulo~$n$, so if $i_{j}=n$ (so $j=r$) then we set $i_{j}+1=1$. By~\cite[page 777]{GG3}, conjugation by $a^{-1}$ permutes cyclically the elements $\rho_{1},\ldots, \rho_{n}, \rho_{1}^{-1},\ldots, \rho_{n}^{-1}$ of $P_{n}(\rp)$, so the $(n-1)\up{th}$ (resp.\ $n\up{th}$) entry of $x'=a^{-(n-1-i_{j})} x a^{(n-1-i_{j})}$ is equal to $\overline{1}$ (resp.\ $\overline{0}$), and $x'\in N$ because $N$ is normal in $B_{n}(\rp)$. Using the relation $b=\sigma_{n-1}a$, we determine the conjugates of the $\rho_{i}$ by $b^{-1}$:
\begin{align*}
b^{-1}\rho_i b &=a^{-1}\sigma_{n-1}^{-1}\rho_{i}\sigma_{n-1}a=a^{-1}\rho_{i}a= \rho_{i+1} \quad\text{for all $1\leq i\leq n-2$}\\
b^{-1}\rho_{n-1}b &=a^{-1}\sigma_{n-1}^{-1}\rho_{n-1}\sigma_{n-1}a=
a^{-1}\sigma_{n-1}^{-1}\rho_{n-1}\sigma_{n-1}^{-1}\ldotp \sigma_{n-1}^{2}a\\
&= a^{-1}\rho_{n} a\ldotp a^{-1}\sigma_{n-1}^{2}a=\rho_{1}^{-1} \ldotp a^{-1} \sigma_{n-1}^{2}a,
\end{align*}
where we have used the relations $\rho_{i}\sigma_{n-1}=\sigma_{n-1}\rho_{i}$ if $1\leq i\leq n-2$ and $\sigma_{n-1}^{-1}\rho_{n-1}\sigma_{n-1}^{-1}=\rho_{n}$ of Van Buskirk's presentation of $B_{n}(\rp)$, as well as the effect of conjugation by $a^{-1}$ on the $\rho_{j}$. Now $\sigma_{n-1}^{2}=A_{n-1,n}\in K_{n}$ by \repr{abeliota}, so $a^{-1} \sigma_{n-1}^{2}a \in K_{n}$ by \rerems{propsK}(\ref{it:propsKa}), and hence $\iota_{\#}(b^{-1}\rho_{n-1}b)=(\overline{1}, \overline{0},\ldots,\overline{0})$. It then follows that $\iota_{\#}(a^{-1}x'a)$ and $\iota_{\#}(b^{-1}x'b)$ have the same entries except in the first and last positions, so if $x''=a^{-1}x'a\ldotp b^{-1}x'b$, we have $\iota_{\#}(x'')= (\overline{1}, \overline{0},\ldots,\overline{0},\overline{1})$. Further, $x''\in N$ since $N$ is normal in $B_{n}(\rp)$. Let $n=2m+\epsilon$, where $m\in \N$ and $\epsilon\in \brak{0,1}$. Then setting
\begin{equation*}
z= a^{-\epsilon}x''a^{\epsilon} \cdotp a^{-(2+\epsilon)}x'' a^{2+\epsilon}\cdots a^{-(2(m-1)+\epsilon)} x'' a^{2(m-1)+\epsilon},
\end{equation*}
we see once more that $z\in N$, and $\iota_{\#}(z)= (\overline{1},\ldots,\overline{1})$ if $n$ is even  and $\iota_{\#}(z)=(\overline{1},\ldots,\overline{1}, \overline{0})$ if $n$ is odd, which completes the proof of the existence of $z$, and thus that of \repr{prop3}(\ref{it:prop3b}).\qedhere
\end{enumerate}
\end{proof}

We end this section by proving \reth{birman}.

\begin{proof}[Proof of \reth{birman}]
Let $S=\St$ or $\rp$. If $n=1$ then $\iota_{\#}$ is an isomorphism and $\im{j_{\#}\left\lvert_{P_{n}}\right.\!}$ is trivial so the result holds. If $n=2$ and $S=\St$ then $P_{n}(\St)$ is trivial, and there is nothing to prove. Now suppose that $S=\St$ and $n\geq 3$. As we mentioned in the introduction, $\ker{\iota_{\#}}=P_{n}(\St)$. Let $(A_{i,j})_{1\leq i<j\leq n}$ be the generating set of $P_{n}$, where $A_{i,j}$ has a geometric representative similar to that given in Figure~\ref{fig:gens}. It is well known that the image of this set by $j_{\#}$ yields a generating set for $P_{n}(\St)$~(\emph{cf.}~\cite[page~616]{Sc}), so $j_{\#}\left\lvert_{P_{n}}\right.$ is surjective, and the statement of the theorem follows. Finally, assume that $S=\rp$ and $n\geq 2$. Once more, $\im{j_{\#}\left\lvert_{P_{n}}\right.}=\setangl{A_{i,j}}{1\leq i<j \leq n}$, and since $A_{i,j}\in \ker{\iota_{\#}}$ by \repr{abeliota}, we conclude that $\ang{\!\!\ang{\,\im{j_{\#}\left\lvert_{P_{n}}\right.\!}}\!\!}_{P_{n}(S)} \subset \ker{\iota_{\#}}$. To prove the converse, first recall from \repr{abeliota} that $\ker{\iota_{\#}}=\Gamma_{2}(P_{n}(\rp))$. Using the standard commutator identities $[x,yz]=[x,y][y,[x,z]] [x,z]$ and $[xy,z]=[x,[y,z]][y,z][x,z]$, $\Gamma_{2}(P_{n}(\rp))$ is equal to the normal closure in $P_{n}(\rp)$ of $\setr{[x,y]}{x,y\in \setl{A_{i,j},\, \rho_{k}}{\text{$1\leq i<j\leq n$ and $1\leq k\leq n$}}}$. It then follows using the relations of \reth{basicpres} that the commutators $[x,y]$ belonging to this set also belong to $\ang{\!\setangl{A_{i,j}}{1\leq i<j \leq n}\!}_{P_{n}(\rp)}$, which is nothing other than $\ang{\!\ang{\,\im{j_{\#}\left\lvert_{P_{n}}\right.\!}}\!}_{P_{n}(S)}$. We conclude by normality that $\ker{\iota_{\#}}\subset \ang{\!\!\ang{\,\im{j_{\#}\left\lvert_{P_{n}}\right.\!}}\!\!}_{P_{n}(S)}$, and this completes the proof of the theorem.
\end{proof}

\section{Some properties of the subgroup $L_{n}$}\label{sec:propsLn}

Let $S=\St$ or $S=\rp$, and for all $m,n\geq 1$, let $\Gamma_{m,n}(S)=P_{m}(S\setminus \brak{x_{1},\ldots,x_{n}})$ denote the $m$-string pure braid group of $S$ with $n$ points removed. In this section, we study $\rpminus$, which is $\Gamma_{n-2,2}(\rp)$, in more detail, and we prove \reth{th4} and \repr{abelianL} that enable us to understand better the structure of the subgroup $L_{n}$ defined in the proof of \repr{prop3}(\ref{it:prop3a})(\ref{it:prop3aii}).

We start by exhibiting a presentation of the group $\Gamma_{m,n}(\rp)$ in terms of the generators of $P_{m+n}(\rp)$ given by \reth{basicpres}. A presentation for $\Gamma_{m,n}(\St)$ is given in~\cite[Proposition~7]{GG11} and will be recalled later in \repr{prespinm}, when we come to proving \reth{prop12}. For $1\leq i<j\leq m+n$, let
\begin{equation}\label{eq:cln}
C_{i,j}= A_{j-1,j}^{-1}\cdots A_{i+1,j}^{-1} A_{i,j} A_{i+1,j} \cdots A_{j-1,j}.
\end{equation}
Geometrically, in terms of Figure~\ref{fig:gens}, $C_{i,j}$ is the image of $A_{i,j}^{-1}$ under the reflection about the straight line segment that passes through the points $x_{1},\ldots, x_{m+n}$. The proof of the following proposition, which we leave to the reader, is similar in nature to that for $\St$, but is a little more involved due to the presence of extra generators that emanate from the fundamental group of $\rp$.

\begin{prop}\label{prop:prej}
Let $n,m\geq 1$. The following constitutes a presentation of the group $\Gamma_{m,n}(\rp)$:\vspace*{-5mm}
\begin{enumerate}
\item[\underline{\textbf{generators:}}] $A_{i,j}$, $\rho_{j}$,  where $1\leq i<j$ and $n+1\leq j \leq m+n$.
\item[\underline{\textbf{relations:}}] 
\mbox{}
\begin{enumerate}[(I)]
\item\label{it:prej1} the Artin relations described by \req{artinaij} among the generators $A_{i,j}$ of $\Gamma_{m,n}(\rp)$.
\item\label{it:prej2} for all $1\leq i<j$ and $n+1\leq j <k\leq m+n$, $A_{i,j}\rho_{k}A_{i,j}^{-1}=\rho_{k}$.
\item\label{it:prej3} for all $1\leq i<j$ and $n+1\leq k<j\leq m+n$,
\begin{equation*}
\rho_{k}A_{i,j}\rho_k^{-1}=
\begin{cases}
A_{i,j}&\text{if $k<i$}\\
\rho_{j}^{-1} C_{i,j}^{-1} \rho_{j} &\text{if $k=i$}\\
\rho_{j}^{-1}C_{k,j}^{-1} \rho_{j}A_{i,j} \rho_{j}^{-1}C_{k,j}\rho_{j} &\text{if $k>i$.}
\end{cases}
\end{equation*}
\item\label{it:prej4} for all $n+1 \leq k<j\leq m+n$, $\rho_{k}\rho_{j}\rho_{k}^{-1}=C_{k,j}\rho_{j}$.
\item\label{it:prej5} for all $n+1\leq j\leq m+n$,
\begin{equation*}
\text{$\rho_{j}\left(\prod_{i=1}^{j-1}\; A_{i,j}\right) \rho_{j}=\left( \prod_{l=j+1}^{m+n}\; A_{j,l}\right)$.}
\end{equation*}
\end{enumerate}
\end{enumerate}
The elements $C_{i,j}$ and $C_{k,j}$ appearing in relations~(\ref{it:prej3}) and~(\ref{it:prej4}) should be rewritten using \req{cln}.
\end{prop}

In the rest of this section, we shall assume that $n=2$, and we shall focus our attention on the groups $\Gamma_{m,2}(\rp)$, where $m\geq 1$, that we interpret as subgroups of $P_{m+2}(\rp)$ via the short exact sequence~\reqref{fnpnp2}. Before proving \reth{th4} and \repr{abelianL}, we introduce some notation that will be used to study the subgroups $K_{n}$ and $L_{n}$. Let $m\geq 2$, and consider the following Fadell-Neuwirth short exact sequence:
\begin{equation}\label{eq:fns}
1\!\to \Omega_{m+1}\to \rpminus[m] \stackrel{r_{m+1}}{\to} \rpminus[m-1]\to\! 1,
\end{equation}
where $r_{m+1}$ is given geometrically by forgetting the last string, and where $\Omega_{m+1}=\pi_{1}(\rp\setminus \brak{x_{1},\ldots,x_{m+1}}, x_{m+2})$. From the Fadell-Neuwirth short exact sequences of the form of~\req{fnpnp2}, $r_{m+1}$ is the restriction of $\map{q_{(m+1)\#}}{P_{m+2}(\rp)}[P_{m+1}(\rp)]$ to $\ker{q_{2\#}}$. The kernel $\Omega_{m+1}$ of $r_{m+1}$ is a free group of rank $m+1$ with a basis $\mathcal{B}_{m+1}$ being given by:
\begin{equation}\label{eq:gensomega}
\mathcal{B}_{m+1}=\setl{A_{k,m+2}, \rho_{m+2}}{1\leq k\leq m}.
\end{equation}
The group $\Omega_{m+1}$ may also be described as the subgroup of $\rpminus[m]$ generated by $\brak{A_{1,m+2},\ldots,A_{m+1,m+2}, \rho_{m+2}}$ subject to the relation:
\begin{equation}\label{eq:surfaceaij}
A_{m+1,m+2}=A_{m,m+2}^{-1}\cdots A_{1,m+2}^{-1}\rho_{m+2}^{-2},
\end{equation}
obtained from relation~(\ref{it:prej5}) of \repr{prej}. Equations~\reqref{cln} and~\reqref{surfaceaij} imply notably that $A_{l,m+2}$ and $C_{l,m+2}$ belong to $\Omega_{m+1}$ for all $1\leq l\leq m+1$. 
Using geometric methods, for $m\geq 2$, we proved the existence of a section
\begin{equation*}
\map{s_{m+1}}{\rpminus[m-1]}[{\rpminus[m]}]
\end{equation*}
for $r_{m+1}$ in~\cite[Theorem~2(a)]{GG7}. Applying induction to \req{fns}, it follows that for all $m\geq 1$:
\begin{equation}\label{eq:semifn}
\rpminus[m]\cong \Omega_{m+1}\rtimes (\Omega_{m}\rtimes (\cdots \rtimes(\Omega_{3}\rtimes \Omega_{2})\cdots)).
\end{equation}
So $\rpminus[m]\cong \F[m+1]\rtimes (\F[m]\rtimes (\cdots \rtimes(\F[3]\rtimes \F[2])\cdots))$, which may be interpreted as the Artin combing operation for $\rpminus[m]$. It follows from this and \req{gensomega} that $\rpminus[m]$ admits $\mathcal{X}_{m+2}$ as a generating set, where:
\begin{equation}\label{eq:genspnminus2}
\mathcal{X}_{m+2}=\setl{A_{i,j},\, \rho_{j}}{3\leq j\leq m+2,\, 1\leq i\leq j-2}.
\end{equation}

\begin{rem}
For what follows, we will need to know an explicit section $s_{m+1}$ for $r_{m+1}$. Such a section may be obtained as follows: for $m\geq 2$, consider the homomorphism $\rpminus[m]\to \rpminus[m-1]$ given by forgetting the  string based at $x_{3}$. By~\cite[Theorem~2(a)]{GG7}), a geometric section is obtained by doubling the second (vertical) string, so that there is a new third string, and renumbering the following strings, which gives rise to an algebraic section for the given homomorphism of the form:
\begin{align*}
A_{i,j} & \mapsto \begin{cases}
A_{1,j+1} & \text{if $i=1$}\\
A_{2,j+1} A_{3,j+1} & \text{if $i=2$}\\
A_{i+1,j+1} & \text{if $3\leq i<j$}
\end{cases}\\
\rho_{j} &\mapsto \rho_{j+1},
\end{align*}
for all $3\leq j\leq m+1$. However, in view of the nature of $r_{m+1}$, we would like this new string to be in the $(m+2)\up{th}$ position. We achieve this by composing the above algebraic section with conjugation by $\sigma_{m+1}\cdots \sigma_{3}$, which gives rise to a section
\begin{equation*}
\map{s_{m+1}}{\rpminus[m-1]}[{\rpminus[m]}]
\end{equation*}
for $r_{m+1}$ that is defined by: 
\begin{equation}\label{eq:secexplicit}
\left\{\begin{aligned}
s_{m+1}(A_{i,j}) &= \begin{cases}
A_{j,m+2} A_{1,j} A_{j,m+2}^{-1} & \text{if $i=1$}\\
A_{j,m+2} A_{2,j} & \text{if $i=2$}\\
A_{i,j} & \text{if $3\leq i <j$}
\end{cases}\\
s_{m+1}(\rho_{j}) &= \rho_{j} A_{j,m+2}^{-1}.
\end{aligned}\right.
\end{equation}
for all $1\leq i< j$ and $3\leq j\leq m+1$.  A long but straightforward calculation using the presentation of $\rpminus[m]$ given by \repr{prej} shows that $s_{m+1}$ does indeed define a section for $r_{m+1}$. 
\end{rem}

We now prove \reth{th4}, which allows us to give a more explicit description of $L_{n}$.

\begin{proof}[Proof of \reth{th4}]
Let $n\geq 3$. By the commutative diagram~\reqref{fnpnp2ext} of short exact sequences, the restriction of the homomorphism $\map{q_{2\#}}{P_{n}(\rp)}[P_{2}(\rp)]$ to $K_{n}$ factors through the inclusion $\ang{\ft[2]}\to P_{2}(\rp)$, and the kernel $L_{n}$ of $q_{2\#}\!\left\vert_{K_{n}}\right.$ is contained in the group $\rpminus$. We may then add a third row to this diagram:
\begin{equation}\label{eq:fnpnp3}
\begin{tikzcd}[ampersand replacement=\&]
 \& 1 \arrow{d} \& 1\arrow{d} \& 1\arrow{d} \& \\
1\arrow{r} \& L_{n} \arrow{r}\arrow{d} \& K_{n}  \arrow{r}{q_{2\#}\left\vert_{K_{n}}\right.} \arrow{d} \& \ang{\ft[2]} \arrow{r}\arrow{d} \& 1\\
1\arrow{r} \& \rpminus \arrow{r} \ar[dashed]{d}{\widehat{\iota}_{n-2}} \& P_{n}(\rp) \arrow{r}{q_{2\#}} \arrow{d}{\iota_{n\#}} \& P_{2}(\rp)  \arrow{r} \arrow{d}{\iota_{2\#}} \& 1\\
1 \arrow{r} \&\Z_{2}^{n-2} \arrow{r}{j} \ar[dashed]{d} \&\Z_{2}^{n}\arrow{r}{\widehat{q}_{2}}\arrow{d} \&\Z_{2}^{2} \arrow{r} \arrow{d} \& 1,\\
\& 1 \& 1 \& 1 \&
\end{tikzcd}
\end{equation}
where $\map{\widehat{q}_{2}}{\Z_{2}^{n}}[\Z_{2}^{2}]$ is projection onto the first two factors, and $\map{j}{\Z_{2}^{n-2}}[\Z_{2}^{n}]$ is the monomorphism defined by $j(\overline{\epsilon_{1}}, \ldots, \overline{\epsilon_{n-2}})= (\overline{0}, \overline{0},\overline{\epsilon_{1}}, \ldots, \overline{\epsilon_{n-2}})$. The commutativity of diagram~\reqref{fnpnp3} thus induces a homomorphism $\map{\widehat{\iota}_{n-2}}{\rpminus}[\Z_{2}^{n-2}]$ that is the restriction of $\iota_{n\#}$ to $\rpminus$ that makes the bottom left-hand square commute. To see that $\widehat{\iota}_{n-2}$ is surjective, notice that if $x\in \Z_{2}^{n-2}$ then the first two entries of $j(x)$ are equal to $\overline{0}$, and using \req{rhotauk}, it follows that there exist $3\leq i_{1}<\cdots < i_{r}\leq n$ such that $\iota_{n\#}(\rho_{i_{1}}\cdots \rho_{i_{r}})=j(x)$. Furthermore, $\rho_{i_{1}}\cdots \rho_{i_{r}}\in \ker{q_{2\#}}$, and by commutativity of the diagram, we also have $\iota_{n\#}(\rho_{i_{1}}\cdots \rho_{i_{r}})= j\circ \widehat{\iota}_{n-2}\,(\rho_{i_{1}}\cdots \rho_{i_{r}})$, whence $x=\widehat{\iota}_{n-2}\,(\rho_{i_{1}}\cdots \rho_{i_{r}})$ by injectivity of $j$. It remains to prove exactness of the first column. The fact that $L_{n}\subset \ker{\widehat{\iota}_{n-2}}$ follows easily. Conversely, if $x\in \ker{\widehat{\iota}_{n-2}}$ then $x\in \rpminus$, and $x\in K_{n}$ by commutativity of the diagram, so $x\in L_{n}$. This proves the first two assertions of the theorem.

To prove the last part of the statement of the theorem, let $m\geq 1$, and consider \req{fns}. 
Since $\widehat{\iota}_{m}$ is the restriction of $\iota_{(m+2)\#}$ to $\rpminus[m]$, we have $\widehat{\iota}_{m}\,(\rho_{j})=(\overline{0},\ldots, \overline{0}, \underbrace{\overline{1}}_{\mathclap{\text{$(j-2)\up{nd}$ position}}}, \overline{0}, \ldots,\overline{0})$ and $\widehat{\iota}_{m}\,(A_{i,j})=(\overline{0}, \ldots,\overline{0})$ for all $1\leq i<j$ and $3\leq j\leq m+2$. So for each $2\leq l\leq m+1$, $\widehat{\iota}_{m}$ restricts to a surjective homomorphism $\map{\widehat{\iota}_{m}\!\left\lvert_{\Omega_{l}}\right.}{\Omega_{l}}[\Z_{2}]$ of each of the factors of \req{semifn}, $\Z_{2}$ being the $(l-1)\up{st}$ factor of $\Z_{2}^{m}$, and using \req{gensomega}, we see that $\keromega[l]{m}$ is a free group of rank $2l-1$ with basis $\widehat{\mathcal{B}}_{l}$ given by:
\begin{equation}\label{eq:basiskeromega}
\widehat{\mathcal{B}}_{l}=\setl{A_{k,l+1}, \rho_{l+1} A_{k,l+1}\rho_{l+1}^{-1}, \rho_{l+1}^{2}}{1\leq k\leq l-1}.
\end{equation}
As we shall now explain, for all $m\geq 2$, the short exact sequence~\reqref{fns} may be extended to a commutative diagram of short exact sequences as follows:
\begin{equation}\label{eq:fnpnp2exta}
\begin{tikzcd}[cramped, sep=scriptsize]
{}& 1\arrow{d} & 1 \arrow{d} & 1 \arrow{d} &\\
1 \arrow{r} & \keromega[m+1]{m} \arrow{d} \arrow{r} & L_{m+2} \arrow{d} \arrow{r}{r_{m+1}\left\lvert_{L_{m+2}}\right.} & L_{m+1} \arrow{d} \arrow{r} & 1\\
1\arrow{r}& \Omega_{m+1}\arrow{r} \arrow{d}{\widehat{\iota}_{m}\left\lvert_{\Omega_{m+1}}\right.} &\rpminus[m] \arrow[yshift=0.5ex]{r}{r_{m+1}} \arrow{d}{\widehat{\iota}_{m}} &\rpminus[m-1] \arrow{r} \arrow{d}{\widehat{\iota}_{m-1}}\arrow[yshift=-0.5ex,dashrightarrow]{l}{s_{m+1}} & 1\\
1 \arrow{r} & \Z_{2} \arrow{d} \arrow{r} & \Z_{2}^{m} \arrow{d} \arrow{r} & \Z_{2}^{m-1} \arrow{d} \arrow{r} & 1.\\
& 1 & 1 & 1 &
\end{tikzcd}
\end{equation}
To obtain this diagram, we start with the commutative diagram that consists of the second and third rows and the three columns (so \emph{a priori}, the arrows of the first row are missing). The commutativity implies that $r_{m+1}$ restricts to the homomorphism $\map{r_{m+1}\left\lvert_{L_{m+2}}\right.}{L_{m+2}}[L_{m+1}]$, which is surjective, since if $w\in L_{m+1}$ is written in terms of the elements of $\mathcal{X}_{m+1}$ then the same word $w$, considered as an element of the group $\rpminus[m]$, belongs to $L_{m+2}$, and satisfies $r_{m+1}(w)=w$. Then the kernel of $r_{m+1}\left\lvert_{L_{m+2}}\right.$, which is also the kernel of $\widehat{\iota}_{m}\left\lvert_{\Omega_{m+1}}\right.$, is equal to $L_{m+2} \cap \Omega_{m+1}$. This establishes the existence of the complete commutative diagram~\reqref{fnpnp2exta} of short exact sequences. By induction, it follows from~\reqref{basiskeromega} and~\reqref{fnpnp2exta} that for all $m\geq 1$, $L_{m+2}$ is generated by
\begin{equation}\label{eq:genskeromega}
\widehat{\mathcal{X}}_{m+2}=\bigcup_{j=3}^{m+2} \widehat{\mathcal{B}}_{j-1} =\setl{A_{i,j},\, \rho_{j}A_{i,j}\rho_{j}^{-1},\, \rho_{j}^{2}}{3\leq j\leq m+2,\, 1\leq i\leq j-2}.
\end{equation}
Using the section $s_{m+1}$ defined by \req{secexplicit}, we see that $s_{m+1}(x)\in L_{m+2}$ for all $x\in \widehat{\mathcal{X}}_{m+1}$, and thus $s_{m+1}$ restricts to a section $\map{s_{m+1}\left\lvert_{L_{m+1}}\right.}{L_{m+1}}[L_{m+2}]$ for $r_{m+1}\left\lvert_{L_{m+2}}\right.$. We conclude by induction on the first row of~\reqref{fnpnp2exta} that:
\begin{align}
L_{m+2}&\cong \keromega[m+1]{m} \rtimes L_{m+1}\label{eq:semiL}\\
&\cong \keromega[m+1]{m} \rtimes \bigl(\keromega[m]{m} \rtimes \bigl(\cdots \rtimes\bigl(\keromega[3]{m} \rtimes \keromega[2]{m}\bigr)\cdots\bigr)\bigr),\label{eq:semiLlong}
\end{align}
the actions being induced by those of \req{semifn}, so by \req{basiskeromega}, $L_{m+2}$ is isomorphic to a repeated semi-direct product of the form $\F[2m+1]\rtimes (\F[2m-1]\rtimes (\cdots \rtimes(\F[5]\rtimes \F[3])\cdots))$. The last part of the statement of \reth{th4} follows by taking $m=n-2$.
%
\end{proof}

A finer analysis of the actions that appear in equations~\reqref{semifn} and~\reqref{semiLlong} now allows us to determine the Abelianisations of $\rpminus$ and $L_{n}$. 

\begin{proof}[Proof of \repr{abelianL}]
If $n=3$ then the two assertions are clear. So assume by induction that they hold for some $n\geq 3$. From the split short exact sequence~\reqref{fns} and \req{semiL} with $m=n-1$, we have:
\begin{equation}\label{eq:semiprod}
\left\{
\begin{aligned}
\rpminus[n-1] &\cong\Omega_{n} \rtimes_{\psi} \rpminus \quad\text{and}\\
L_{n+1} &\cong \keromega{n-1} \rtimes_{\psi} L_{n},
\end{aligned}\right.
\end{equation}
where $\psi$ denotes the action given by the section $s_{n}$, and the action induced by the restriction of the section $s_{n}$ to $L_{n}$ respectively. 

Before going any further, we recall some general considerations\label{genconsid} from the paper~\cite[pages~3387--88]{GG6} concerning the Abelianisation of semi-direct products. If $H$ and $K$ are groups, and if $\map{\phi}{H}[\aut{K}]$ is an action of $H$ on $K$ then one may deduce easily from~\cite[Proposition~3.3]{GG6} that:
\begin{equation}\label{eq:absemi}
(K\rtimes_{\phi} H)\up{Ab}\cong \Delta(K)\oplus H\up{Ab},
\end{equation}
where:
\begin{equation*}
\Delta(K) =K\bigl/\bigl\langle\Gamma_{2}(K)\cup \widehat{K}\bigr\rangle\bigr. \quad\text{and}\quad
\widehat{K}=\setangl{\phi(h)(k) \cdot k^{-1}}{\text{$h\in H$ and $k\in K$}}.
\end{equation*}
Recall that $\widehat{K}$ is normal in $K$ (\emph{cf.}~\cite[lines~1--4, page~3388]{GG6}), so $\langle\Gamma_{2}(K)\cup \widehat{K}\bigr\rangle $ is normal in $ K$. If $k\in K$, let $\overline{k}$ denote its image under the canonical projection $K\to \Delta(K)$. For all $k,k'\in K$ and $h,h'\in H$, we have:
\begin{align}
\phi(hh')(k)\cdot k^{-1}&= \phi(h)(\phi(h')(k)) \cdot \phi(h')(k^{-1}) \cdot \phi(h')(k)\cdot k^{-1}\notag\\
&=\phi(h)(k'') \cdot k''^{-1} \cdot \phi(h')(k)\cdot k^{-1}\label{eq:hhprime}\\
\phi(h)(kk')\cdot (kk')^{-1}&= \bigl(\phi(h)(k)\cdot k^{-1}\bigr)\cdot k\bigl(\phi(h)(k')\cdot k'^{-1}\bigr)k^{-1}.\label{eq:kkprime}
\end{align}
where $k''=\phi(h')(k)$ belongs to $K$. Let $\mathcal{H}$ and $\mathcal{K}$ be generating sets for $H$ and $K$ respectively. By induction on word length relative to the elements of $\mathcal{H}$, \req{hhprime} implies that $\widehat{K}$ is generated by elements of the form $\phi(h)(k)\cdot k^{-1}$, where $h\in \mathcal{H}$ and $k\in K$. A second induction on word length relative to the elements of $\mathcal{K}$ and \req{kkprime} implies that $\widehat{K}$ is normally generated by the elements of the form $\phi(h)(k)\cdot k^{-1}$, where $h\in \mathcal{H}$ and $k\in \mathcal{K}$. By standard arguments involving group presentations, since $\Gamma_{2}(K) \subset \bigl\langle\Gamma_{2}(K)\cup \widehat{K}\bigr\rangle\bigr.$, $\Delta(K)$ is Abelian, and a presentation of $\Delta(K)$ may be obtained by Abelianising a given presentation of $K$, and by adjoining the relators of the form $\overline{\phi(h)(k) \cdot k^{-1}}$,\label{genconsid2} where $h\in \mathcal{H}$ and $k\in \mathcal{K}$.

We now take $K=\Omega_{n}$ (resp.\ $K=\keromega{n-1}$), $H=\rpminus$ (resp.\ $H= L_{n}$) and $\phi=\psi$. Applying the induction hypothesis and \req{absemi} to \req{semiprod}, to prove parts~(\ref{it:abparta}) and~(\ref{it:abpartb}), it thus suffices to show that:
\begin{align}
\Delta(\Omega_{n}) &\cong \Z^{2}, \quad\text{and that}\label{eq:isof1}\\
\Delta\left(\keromega{n-1}\right) & \cong \Z^{2n-1} \label{eq:isof2}
%
\end{align}
respectively. We first establish the isomorphism~\reqref{isof1}. As we saw previously, $\Delta(\Omega_{n})$ is Abelian, and to obtain a presentation of $\Delta(\Omega_{n})$, we add the relators of the form $\overline{\psi(\tau)(\omega)\cdot \omega^{-1}}$ to a presentation of $(\Omega_{n})\up{Ab}$, where $\tau\in \mathcal{X}_{n}$ and $\omega\in \mathcal{B}_{n}$. In $\Delta(\Omega_{n})$, such relators may be written as:
\begin{equation}\label{eq:newrels}
\overline{s_{n}(\tau) \omega  (s_{n}(\tau))^{-1} \omega^{-1}}= \overline{s_{n}(\tau) \omega  (s_{n}(\tau))^{-1}} \;\overline{\omega^{-1}}.
\end{equation}
We claim that it is not necessary to know explicitly the section $s_{n}$ in order to determine these relators. Indeed, for all $\tau\in \mathcal{X}_{n}$, we have $p_{n+1}(\tau)=\tau$; note that we abuse notation here by letting $\tau$ also denote the corresponding element of $\mathcal{X}_{n+1}$ in $\rpminus[n-1]$. Thus $s_{n}(\tau) \tau^{-1}\in \ker{p_{n+1}}$, and hence there exists $\omega_{\tau}\in \Omega_{n}$ such that $s_{n}(\tau)=\omega_{\tau} \tau$. Since $\Delta(\Omega_{n})$ is Abelian, it follows that:
\begin{equation*}
\overline{s_{n}(\tau) \omega (s_{n}(\tau))^{-1}}= \overline{\omega_{\tau} \tau \omega \tau^{-1}\omega_{\tau}^{-1}}=\overline{\vphantom{\omega_{\tau}^{-1}}\omega_{\tau}}\; \overline{\vphantom{\omega_{\tau}^{-1}}\tau \omega \tau^{-1}} \; \overline{\omega_{\tau}^{-1}}= \overline{\tau \omega \tau^{-1}},
\end{equation*}
and thus the relators of \req{newrels} become:
\begin{equation}\label{eq:relatortau}
\overline{s_{n}(\tau) \omega  (s_{n}(\tau))^{-1} \omega^{-1}}=\overline{\tau \omega \tau^{-1}} \; \overline{\omega^{-1}}.
\end{equation}
This proves the claim. In what follows, the relations~(\ref{it:prej1})--(\ref{it:prej5}) refer to those of the presentation of $\rpminus[n-1]$ given by \repr{prej}. Using this presentation 
and the fact that $\Delta(\Omega_{n})$ is Abelian, we see immediately that $\overline{\tau \omega \tau^{-1}}=\overline{\omega}$ for all $\tau\in \mathcal{X}_{n}$ and $\omega\in \mathcal{B}_{n}$, with the following exceptions:
\begin{enumerate}[(i)]
\item\label{it:relators1} $\tau=\rho_{j}$ and $\omega=A_{j,n+1}$  for all $3\leq j\leq n-1$. Then $\overline{\rho_{j}A_{j,n+1}\rho_{j}^{-1}}=\overline{C_{j,n+1}^{-1}}=\overline{A_{j,n+1}^{-1}}$, using relation~(\ref{it:prej3}) and \req{cln}, which yields the relator $\bigl(\,\overline{A_{j,n+1}}\,\bigr)^{2}$ in $\Delta(\Omega_{n})$.
\item\label{it:relators2} $\tau=\rho_{j}$ and $\omega=\rho_{n+1}$  for all $3\leq j\leq n$. Then $\overline{\rho_{j}\rho_{n+1}\rho_j^{-1}}=\overline{C_{j,n+1} \rho_{n+1}}=\overline{A_{j,n+1}}\;\overline{\vphantom{A_{j,n+1}}\rho_{n+1}}$ by relation~(\ref{it:prej4}) and \req{cln}, which yields the relator $\overline{A_{j,n+1}}$ in $\Delta(\Omega_{n})$.
\end{enumerate}
The relators of~(\ref{it:relators2}) above clearly give rise to those of~(\ref{it:relators1}). To obtain a presentation of $\Delta(\Omega_{n})$, which by \req{gensomega} is an Abelian group with generating set
\begin{equation*}
\setl{\overline{A_{l,n+1}},\overline{\rho_{n+1}}}{1\leq l\leq n-1},
\end{equation*}
we must add the relators $\overline{A_{j,n+1}}$ for all $3\leq j\leq n$. Thus for $j=3, \dots, n-1$, the elements $\overline{A_{j,n+1}}$ of this generating set are trivial. Further, $\overline{A_{n,n+1}}$ is also trivial, but by relation~\reqref{surfaceaij}, one of the remaining generators $\overline{A_{j,n+1}}$ may be deleted, $\overline{A_{2,n+1}}$ say, from which we see that $\Delta(\Omega_{n})$ is a free Abelian group of rank two with $\brak{\overline{A_{1,n+1}},\overline{\rho_{n+1}}}$ as a basis. This establishes the isomorphism~\reqref{isof1}, and so proves part~(\ref{it:abparta}).

We now prove part~(\ref{it:abpartb}). As we mentioned previously, it suffices to establish the isomorphism~\reqref{isof2}. Since $\keromega{n-1}$ is a free group of rank $2n-1$, we must thus show that $\Delta(\keromega{n-1})=(\keromega{n-1})\up{Ab}$. We take $K=\keromega{n-1}$ (resp.\ $H=L_{n-2}$) to be equipped with the basis $\widehat{\mathcal{B}}_{n}$ (resp.\ the generating set $\widehat{\mathcal{X}}_{n}$) of \req{basiskeromega} (resp.\ of \req{genskeromega}). The fact that $\keromega{n-1} $ is normal in $\Omega_{n}$ implies that $A_{l,n+1}$, $\rho_{n+1} A_{l,n+1}\rho_{n+1}^{-1}$, $C_{l,n+1}$ and $\rho_{n+1} C_{l,n+1}\rho_{n+1}^{-1}$ belong to $\keromega{n-1}$ for all $1\leq l\leq n$ by equations~\reqref{cln} and~\reqref{surfaceaij}. Repeating the argument given between equations~\reqref{newrels} and~\reqref{relatortau}, we see that \req{relatortau} holds for all $\tau\in \widehat{\mathcal{X}}_{n}$ and $\omega\in \widehat{\mathcal{B}}_{n}$, where $\overline{\omega}$ now denotes the image of $\omega$ under the canonical projection $\keromega{n-1}\to \Delta(\keromega{n-1})$. For $\alpha\in \rpminus$, let $c_{\alpha}$ denote conjugation in $\keromega{n-1}$ by $\alpha$ (which we consider to be an element of $\rpminus[n-1]$).  The automorphism $c_{\alpha}$ is well defined because $\keromega{n-1}=\Omega_{n} \cap L_{n-1}$, so that $\keromega{n-1}$ is normal in $ \rpminus[n-1]$. We claim that $\bigl\langle\Gamma_{2}(K)\cup \widehat{K}\bigr\rangle\bigr.$ is invariant under $c_{\alpha}$. To see this, note first that $\Gamma_{2}(K)$ is clearly invariant since it is a characteristic subgroup of $K$. On the other hand, suppose that $\omega\in \keromega{n-1}$, $\tau\in L_{n-2}$ and $\alpha\in \rpminus$. Since $s_{n}(\tau)\in L_{n-1}$, $L_{n-1}$ is normal in $ \rpminus[n-1]$ and $L_{n-2} $ is normal in $ \rpminus$, we have $\alpha s_{n}(\tau) \alpha^{-1}\in L_{n-1}$, $\tau'=p_{n+1}(\alpha s_{n}(\tau) \alpha^{-1})=\alpha \tau \alpha^{-1}\in L_{n-2}$, and thus $s_{n}(\tau'^{-1})(\alpha s_{n}(\tau) \alpha^{-1})\in \keromega{n-1}$. Hence there exists $\omega_{\tau'}\in \keromega{n-1}$ such that $\alpha s_{n}(\tau) \alpha^{-1}=s_{n}(\tau') \omega_{\tau'}$. Now $\keromega{n-1}$ is normal in $ \rpminus[n-1]$, so $\omega'=\alpha\omega \alpha^{-1}\in \keromega{n-1}$, and therefore:
\begin{align*}
c_{\alpha}\bigl(s_{n}(\tau)\omega s_{n}(\tau^{-1})\omega^{-1}\bigr)&=\alpha\bigl(s_{n}(\tau)\omega s_{n}(\tau^{-1})\omega^{-1}\bigr)\alpha^{-1} = s_{n}(\tau') \omega_{\tau'} \omega' \omega_{\tau'}^{-1} s_{n}(\tau'^{-1}) \omega'^{-1}\\
&= s_{n}(\tau') (\omega_{\tau'} \omega' \omega_{\tau'}^{-1}) s_{n}(\tau'^{-1}) (\omega_{\tau'} \omega'^{-1} \omega_{\tau'}^{-1}) \cdot \omega_{\tau'} \omega' \omega_{\tau'}^{-1} \omega'^{-1},
\end{align*}
which belongs to $\bigl\langle\Gamma_{2}(K)\cup \widehat{K}\bigr\rangle\bigr.$ because $s(\tau') (\omega_{\tau'} \omega' \omega_{\tau'}^{-1}) s(\tau'^{-1}) (\omega_{\tau'} \omega'^{-1} \omega_{\tau'}^{-1})\in \widehat{K}$ and $\omega_{\tau'} \omega' \omega_{\tau'}^{-1} \omega'^{-1} \in \Gamma_{2}(K)$. This proves the claim, and implies that $c_{\alpha}$ induces an endomorphism $\widehat{c}_{\alpha}$ (an automorphism in fact, whose inverse is $\widehat{c}_{\alpha^{-1}}$) of $\Delta(K)$, in particular, if $\alpha,\alpha'\in \rpminus$ and $\omega\in \keromega{n-1}$ then $\overline{\alpha\alpha' \omega \alpha'^{-1} \alpha^{-1}}= \overline{c_{\alpha\alpha'}(\omega)}= \widehat{c}_{\alpha}(\widehat{c}_{\alpha'}(\overline{\omega}))$.

We next compute the elements $\overline{\tau \omega \tau^{-1}}$ of $\Delta(\keromega{n-1})$ in the case where $\tau=A_{i,j}$, $3\leq j\leq n$ and $1\leq i\leq j-2$, and $\omega\in \widehat{\mathcal{B}}_{n}$:
\begin{enumerate}[(i)]
\item\label{it:actionomega1} Let $\omega=A_{l,n+1}$, for $1\leq l\leq n-1$. Then
\begin{align*}
\tau \omega \tau^{-1}&=\begin{cases}
A_{l,n+1} & \text{if $j<l$ or if $l<i$}\\
A_{l,n+1}^{-1} A_{i,n+1}^{-1} A_{l,n+1}A_{i,n+1} A_{l,n+1}& \text{if $j=l$}\\
A_{j,n+1}^{-1}A_{l,n+1}A_{j,n+1}& \text{if $i=l$}\\
A_{j,n+1}^{-1} A_{i,n+1}^{-1} A_{j,n+1}A_{i,n+1} A_{l,n+1}A_{i,n+1}^{-1} A_{j,n+1}^{-1} A_{i,n+1}A_{j,n+1} & \text{if $i<l<j$}
\end{cases}
\end{align*}
by the Artin relations. We thus conclude that $\overline{\tau \omega \tau^{-1}}=\overline{\omega}$ in this case.

\item If $\omega=\rho_{n+1} A_{l,n+1}\rho_{n+1}^{-1}$, where $1\leq l\leq n-1$ then $\tau \omega \tau^{-1}= \rho_{n+1} (A_{i,j}A_{l,n+1}A_{i,j}^{-1})\rho_{n+1}^{-1}$, and from case~(\ref{it:actionomega1}), we deduce also that $\overline{\tau \omega \tau^{-1}}=\overline{\omega}$.

\item Let $\omega=\rho_{n+1}^{2}$. Then $\tau \omega \tau^{-1}=\omega$, hence $\overline{\tau \omega \tau^{-1}}=\overline{\omega}$.
\end{enumerate}
So if $\tau=A_{i,j}$ then the relators given by \req{relatortau} are trivial for all $\omega\in \widehat{\mathcal{B}}_{n}$, and $\widehat{c}_{A_{i,j}}=\id_{\Delta\left(\keromega{n-1}\right)}$.

Now suppose that $\tau= \rho_{j}A_{i,j} \rho_{j}^{-1}$, where $3\leq j\leq n$ and $1\leq i\leq j-2$. Then for all $\omega\in \widehat{\mathcal{B}}_{n}$, we have:
\begin{equation*}
\overline{\tau \omega\tau^{-1}}=\overline{c_{\tau}(\omega)}= \widehat{c}_{\rho_{j}}\circ \widehat{c}_{A_{i,j}} \circ \widehat{c}_{\rho_{j}^{-1}}(\overline{\omega})=\overline{\omega},
\end{equation*}
since $\widehat{c}_{A_{i,j}}=\id_{\Delta\left(\keromega{n-1}\right)}$, so $\widehat{c}_{\rho_{j}A_{i,j} \rho_{j^{-1}}}=\id_{\Delta\left(\keromega{n-1}\right)}$.

By \req{genskeromega}, it remains to study the elements of the form $\overline{\tau \omega\tau^{-1}}$, where $\tau=\rho_{j}^2$, $3\leq j\leq n$, and $\omega\in \widehat{\mathcal{B}}_{n}$. Since $\overline{\rho_{j}^2 \omega\rho_{j}^{-2}}= \widehat{c}_{\rho_{j}}^{\mskip 4mu 2}(\overline{\omega})$, we first analyse $\widehat{c}_{\rho_{j}}$.
\begin{enumerate}
\item\label{it:actionker4} If $\omega=A_{l,n+1}$, for $1\leq l\leq n-1$ then by relation~(\ref{it:prej3}) and equations~\reqref{cln} and~\reqref{surfaceaij}, we have:
\begin{align*}
\widehat{c}_{\rho_{j}}(\overline{\omega})&= \widehat{c}_{\rho_{j}}(\overline{A_{l,n+1}})=\overline{\rho_{j} A_{l,n+1} \rho_{j}^{-1}}\\
&=\begin{cases}
\overline{A_{l,n+1}}&\text{if $j<l$}\\
\overline{\rho_{n+1}^{-2} \cdot \rho_{n+1} C_{l,n+1}^{-1}\rho_{n+1}^{-1} \cdot \rho_{n+1}^2}
&\text{if $j=l$}\\
\overline{\rho_{n+1}^{-2} \cdot \rho_{n+1} C_{j,n+1}^{-1} \rho_{n+1}^{-1} \cdot \rho_{n+1}^2 \cdot A_{l,n+1} \cdot \rho_{n+1}^{-2} \cdot \rho_{n+1}  C_{j,n+1} \rho_{n+1}^{-1} \cdot \rho_{n+1}^2}
&\text{if $j>l$}
\end{cases}\\
&= \begin{cases}
\overline{A_{l,n+1}} & \text{if $j\neq l$}\\
\overline{\rho_{n+1} C_{j,n+1}^{-1}\rho_{n+1}^{-1}}=\left(\,\overline{\rho_{n+1} A_{j,n+1}\rho_{n+1}^{-1}}\,\right)^{-1} & \text{if $j=l$.}
\end{cases}
\end{align*}

\item\label{it:actionker5} Let $\omega=\rho_{n+1} A_{l,n+1}\rho_{n+1}^{-1}$, where $1\leq l\leq n-1$. Relation~(\ref{it:prej4}) implies that $\rho_{j} \rho_{n+1} \rho_{j}^{-1}=C_{j,n+1} \rho_{n+1}$, and so by case~(\ref{it:actionker4}) above, we have:
\begin{equation*}
\widehat{c}_{\rho_{j}}(\overline{\omega})=  \widehat{c}_{\rho_{j}}\left(\,\overline{\rho_{n+1} A_{l,n+1}\rho_{n+1}^{-1}}\,\right)=\begin{cases}
\overline{\rho_{n+1} A_{l,n+1}\rho_{n+1}^{-1}} & \text{if $j\neq l$}\\[0.5ex]
\overline{C_{j,n+1}^{-1}}=\left(\,\overline{A_{j,n+1}}\,\right)^{-1} & \text{if $j=l$.}
\end{cases}
\end{equation*}

\item\label{it:actionker6} Let $\omega=\rho_{n+1}^{2}$. By relation~(\ref{it:prej4}) and equations~\reqref{cln} and~\reqref{surfaceaij}, we have:
\begin{align*}
\widehat{c}_{\rho_{j}}(\overline{\omega})&=\widehat{c}_{\rho_{j}}(\overline{\rho_{n+1}^{2}})=\overline{(\rho_{j} \rho_{n+1}\rho_{j}^{-1})^2}=\overline{\rho_{n+1} C_{j,n+1}\rho_{n+1}^{-1} \cdot \rho_{n+1}^{2} \cdot C_{j,n+1}}\\
&= \overline{\rho_{n+1} A_{j,n+1}\rho_{n+1}^{-1} \cdot \rho_{n+1}^{2} \cdot A_{j,n+1}},
\end{align*}
\end{enumerate}
from which we obtain:
\begin{align*}
\widehat{c}_{\rho_{j}^{2}}\Bigl(\,\overline{\rho_{n+1}^{2}}\,\Bigr) &= \widehat{c}_{\rho_{j}}\Bigl(\,\overline{\rho_{n+1} A_{j,n+1}\rho_{n+1}^{-1} \cdot \rho_{n+1}^{2} \cdot A_{j,n+1}}\,\Bigr)\\
&= \overline{A_{j,n+1}^{-1}\cdot\rho_{n+1} A_{j,n+1}\rho_{n+1}^{-1} \cdot \rho_{n+1}^{2} \cdot A_{j,n+1}\cdot \bigl(\rho_{n+1} A_{j,n+1}\rho_{n+1}^{-1}\bigr)^{-1}} = \overline{\rho_{n+1}^{2}}.
\end{align*}
So by \req{basiskeromega}, we also have $\widehat{c}_{\rho_{j}^{2}}=\id_{\Delta\left(\keromega{n-1}\right)}$. Hence for all $\tau\in L_{n-2}$ and $\omega\in \keromega{n-1}$, it follows that $\widehat{c}_{\tau}(\overline{\omega})=\overline{\omega}$ , and thus the relators $\overline{\psi(\tau)(\omega)\cdot \omega^{-1}}$ are all trivial. Since a presentation for $\Delta\left(\keromega{n-1}\right)$ is obtained by Abelianising a given presentation of $\keromega{n-1}$ and adjoining these relators, we conclude that $\Delta\left(\keromega{n-1}\right)=\left(\keromega{n-1}\right)\up{Ab}$. In particular, the fact that $\keromega{n-1}$ is a free group of rank $2n-1$ gives rise to the isomorphism~\reqref{isof2}. This completes the proof of the proposition.
\end{proof}

\begin{rems}\mbox{}\label{rem:artin}
\begin{enumerate}[(a)]
\item\label{it:artin1} An alternative description of $\rpminus$, similar to that of \req{semifn}, but with the parentheses in the opposite order, may be obtained as follows. Let $m\geq 2$ and $q\geq 1$, and consider the following Fadell-Neuwirth short exact sequence:
\begin{multline}\label{eq:fnsalt}
1\to P_{m-1}(\rp\setminus \brak{x_{1},\ldots, x_{q+1}})\to P_{m}(\rp\setminus\brak{x_{1},\ldots, x_{q}})\to\\
P_{1}(\rp\setminus\brak{x_{1},\ldots, x_{q}})\to 1,
\end{multline}
given geometrically by forgetting the last $m-1$ strings. Since the quotient is a free group $\F[q]$ of rank $q$, the above short exact sequence splits, and so
\begin{equation*}
P_{m}(\rp\setminus\brak{x_{1},\ldots, x_{q}}) \cong P_{m-1}(\rp\setminus \brak{x_{1},\ldots, x_{q+1}}) \rtimes \F[q],
\end{equation*}
and thus
\begin{equation}\label{eq:semifnbrak}
P_{n-2}(\rp \setminus \brak{x_1, x_2}))\cong (\cdots ((\F[n-1]\rtimes \F[n-2])\rtimes \F[n-3])\rtimes \cdots \rtimes \F[3])\rtimes \F[2].
\end{equation}
by induction. The ordering of the parentheses thus occurs from the left, in contrast with that of \req{semifn}. The decomposition given by \req{semifn} is in some sense stronger than that of \reqref{semifnbrak}, since in the first case, every factor acts on each of the preceding factors, which is not necessarily the case in \req{semifnbrak}, so \req{semifn} engenders a decomposition of the form~\reqref{semifnbrak}. This is a manifestation of the fact that the splitting of the corresponding Fadell-Neuwirth sequence~\reqref{fns} is non trivial, while that of~\reqref{fnsalt} is obvious. 

\item Note that $L_{4}$, which is the kernel of the homomorphism $\map{\widehat{\iota}_{2}}{\rpminus[2]}[\Z_{2}^{2}]$, is also the subgroup of index $4$ of the group $(B_4(\rp))^{(3)}$ that appears in~\cite[Theorem~3(d)]{GG8}. Indeed, 
by~\cite[equation~(127)]{GG8}, this subgroup of index $4$ is isomorphic to the semi-direct product:
\begin{equation*}
\F[5](A_{1,4},A_{2,4},\rho_{4}^{2}, \rho_{4} A_{1,4}\rho_{4}^{-1},\rho_{4} A_{2,4}\rho_{4}^{-1})\rtimes \F[3](A_{2,3},\rho_{3}^{2}, \rho_{3}A_{2,3}\rho_{3}^{-1}),
\end{equation*}
the action being given by~\cite[equations~(129)--(131)]{GG8}  (the element $B_{i,j}$ of~\cite{GG8} is the element $A_{i,j}$ of this paper).
\end{enumerate}
\end{rems}

\begin{rem}\label{rem:Lnuniqueness}
Using the ideas of the last paragraph of the proof of \repr{prop3}(\ref{it:prop3b}), one may show that $L_{n}$ is not normal in $B_n(\rp)$. Although the subgroup $L_{n}$ is not unique with respect to the properties of the statement of \repr{prop3}(\ref{it:prop3a})(\ref{it:prop3aii}), there are only a finite number of subgroups, $2^{n(n-2)}$ to be precise, that satisfy these properties. To prove this, we claim that the set of torsion-free subgroups $L_{n}'$ of $K_{n}$ such that $K_{n}=L_{n}'\oplus \ang{\ft}$ is in bijection with the set $\setr{\ker{f}}{f\in \operatorname{Hom}(L_{n},\Z_{2})}$. To prove the claim, let $K=K_{n}$, $L=L_{n}$, let $\map{q}{K}[K/L]$ be the canonical surjection, and set 
\begin{equation*}
\Delta=\setr{L'}{\text{$L'<K$, $L'$ is torsion free, and $K= L'\oplus \bigl\langle \ft \bigr\rangle$}}.
\end{equation*}
Clearly $L\in \Delta$, so $\Delta\neq \vide$. Consider the map $\map{\phi}{\Delta}[\setr{\ker{f}}{f\in \operatorname{Hom}(L,\Z_{2})}]$ defined by $\phi(L')=L\cap L'$. This map is well defined, since if $L'=L$ then $\phi(L')=L$ is the kernel of the trivial homomorphism of $\operatorname{Hom}(L,\Z_{2})$, and if $L'\neq L$ then $L' \nsubset L$ since $[K:L']=[K:L]=2$, and so $q\left\lvert_{L'}\right.$ is surjective as $K/L\cong \Z_{2}$. Thus $\ker{q\left\lvert_{L'}\right.}=\phi(L')$ is of index $2$ in $L$, in particular, $\phi(L')$ is the kernel of some non-trivial element of $\operatorname{Hom}(L,\Z_{2})$.

We now prove that $\phi$ is surjective. Let $f\in \operatorname{Hom}(L,\Z_{2})$, and set $L''=\ker{f}$. If $f=0$ then $L''=L$, and $\phi(L)=L''$. So suppose that $f\neq 0$. Then $f$ is surjective, and $L''=\ker{f}$ is of index $2$ in $L$. Let $x\in L\setminus L''$. Then
\begin{equation}\label{eq:lprime}
L=L'' \amalg xL'',
\end{equation}
where $\amalg$ denotes the disjoint union. Since $K=L \amalg \ft  L$, it follows that
\begin{equation}\label{eq:klprime}
K=L'' \amalg x L'' \amalg\ft  L''\amalg x\ft  L'',
\end{equation}
where $\amalg$ denotes the disjoint union. Set $L'=L'' \amalg x\ft  L''$. By \req{lprime}, $x^{2}\ft  L''=\ft x^{2} L''=\ft L''$ because $\ft$ is central and of order $2$, and hence $K=L' \amalg xL'$. Using once more \req{lprime}, we see that $L'$ is a group, and so the equality $K=L' \amalg xL'$ implies that $[K:L']=2$. Further, since the only non-trivial torsion element of $K$ is $\ft$, $L'$ is torsion free by \req{klprime}, and so the short exact sequence $1\to L'\to K\to \Z_{2}\to 1$ splits. Thus $L'\in \Delta$, and $\phi(L')=L''$ using equations~\reqref{lprime} and~\reqref{klprime}.

It remains to prove that $\phi$ is injective. Let $L_{1}',L_{2}'\in \Delta$ be such that $L_{1}'\cap L=\phi(L_{1}')=\phi(L_{2}')= L_{2}'\cap L$. If one of the $L_{i}'$, $L_{1}'$ say, is equal to $L$ then we must also have $L_{2}'=L$ because $L\subset L_{2}'$ and $L$ and $L_{2}'$ have the same index in $K$. So suppose that $L_{i}'\neq L$ for all $i\in\brak{1,2}$. If $i\in \brak{1,2}$ then $L''=\phi(L_{i}')=L\cap L_{i}'=\ker{f_{i}}$ for some non-trivial $f_{i}\in \operatorname{Hom}(L,\Z_{2})$, and thus $[L:L'']=2$. Let us show that $L_{1}'\subset L_{2}'$. Let $x\in L_{1}'$. If $x\in L$ then $x\in L''$, so $x\in L_{2}'$, and we are done. So assume that $x\notin L$, and suppose that $x\notin L_{2}'$. Then $q(x)$ is equal to the non-trivial element of $K/L$, and since $K/L\cong \Z_{2}$ and $\ft\notin L$, we see that $x\ft\in L$. Further, $K= L_{2}' \amalg xL_{2}'$ since $[K:L_{2}']=2$, and so $x\ft \in L_{2}'$ (for otherwise $x\ft\in xL_{2}'$, which implies that $\ft\in L_{2}'$, which is impossible because $L_{2}'$ is torsion free). But then $x\ft \in L\cap L_{2}'=L''$, and hence $x\ft \in L_{1}'$. But this would imply that $\ft \in L_{1}'$, which contradicts the fact that $L_{1}'$ is torsion free. We conclude that $L_{1}'\subset L_{2}'$, and exchanging the rôles of $L_{1}'$ and $L_{2}'$, we see that $L_{1}'=L_{2}'$, which proves that $\phi$ is injective, so is bijective, which proves the claim. Therefore the cardinality of $\Delta$ is equal to the order of the group $H^1(L, \Z_2)$, which is equal in turn to that of $H_1(L, \Z_2)$. By \repr{abelianL}(\ref{it:abpartb}), we have $L\up{Ab}=H_1(L, \Z)\cong \Z^{n(n-2)}$, so $H_1(L, \Z_2)\cong H_1(L, \Z) \otimes \Z_{2}\cong \Z_{2}^{n(n-2)}$, and the number of subgroups of $K$ that satisfy the properties of \repr{prop3}(\ref{it:prop3a}) is equal to $2^{n(n-2)}$ as asserted.
\end{rem}

\section{The virtual cohomological dimension of $B_n(S)$ and $P_n(S)$  for $S=\St,\rp$}\label{sec:vcd}

Let $S=\St$ (resp.\ $S=\rp$), and for all $m,n\geq 1$, let $\Gamma_{n,m}(S)=P_{n}(S\setminus \brak{x_{1},\ldots,x_{m}})$ denote the $n$-string pure braid group of $S$ with $m$ points removed. In order to study various cohomological properties of the braid groups of $S$ and prove \reth{prop12}, we shall study $\Gamma_{n,m}(S)$. To prove \reth{prop12} in the case $S=\St$, by \req{pns2sum}, it will suffice to compute the cohomological dimension of $P_{n-3}(\St\setminus \brak{x_{1},x_{2},x_{3}})$. We recall the following presentation of $\Gamma_{n,m}(\St)$ from~\cite{GG11}. The result was stated for $m\geq 3$, but it also holds for $m\leq 2$.

\begin{prop}[{\cite[Proposition~7]{GG11}}]\label{prop:prespinm}
Let $n,m\geq 1$. The following constitutes a presentation of the group $\Gamma_{n,m}(\St)$:
\begin{enumerate}
\item[\underline{\textbf{generators:}}] $A_{i,j}$, where $1\leq i<j$ and $m+1\leq j\leq m+n$.
\item[\underline{\textbf{relations:}}]\mbox{}
\begin{enumerate}[(i)]
\item\label{it:pinmrelsa} the Artin relations described by \req{artinaij} among the generators $A_{i,j}$ of $\Gamma_{n,m}(\St)$.
\item for all $m+1\leq j\leq m+n$,
$\left( \prod_{i=1}^{j-1}\; A_{i,j}\right) \left( \prod_{k=j+1}^{m+n}\; A_{j,k}\right)=1$.
\end{enumerate}
\end{enumerate}
\end{prop}

Let $N$ denote the kernel of the homomorphism $\Gamma_{n,m}(S) \to \Gamma_{n-1,m}(S)$ obtained geometrically by forgetting the last string. If $S=\St$ (resp.\ $S=\rp$) then $N$ is a free group of rank $m+n-2$ (resp.\ $m+n-1$) and is equal to $\ang{A_{1,m+n},A_{2,m+n},\ldots, A_{m+n-1,m+n}}$ (resp.\ $\ang{A_{1,m+n},A_{2,m+n},\ldots, A_{m+n-1,m+n},\rho_{m+n}}$). Clearly $N$ is normal in $\Gamma_{n,m}(S)$. Further, it follows from relations~(\ref{it:pinmrelsa}) of \repr{prespinm} (resp.\ relations~(\ref{it:prej3}) and~(\ref{it:prej4}) of \repr{prej}) that the action by conjugation of $\Gamma_{n,m}(S)$ on $N$ induces (resp.\ does not induce) the trivial action on the Abelianisation of $N$.
In order to determine the virtual cohomological dimension of the braid groups of $S$ and prove \reth{prop12}, we shall compute the cohomological dimension of a torsion-free finite-index subgroup. In the case of $\St$ (resp.\ $\rp$), we choose the subgroup $\Gamma_{n-3,3}(\St)$ that appears in the decomposition given in \req{pns2sum} (resp.\ the subgroup $\Gamma_{n-2,2}(\rp)$ that appears in \req{fnpnp2}).



\begin{proof}[Proof of \reth{prop12}]
Let $S=\St$ (resp.\ $S=\rp$), let $n>3$ and $k=3$ (resp.\ $n>2$ and $k=2$), and let $k\leq m <n$. Then by \req{pns2sum} (resp.\ \req{fnpnp2}) and \req{sessym}, $\Gamma_{n-m,m}(S)$ is a subgroup of finite index of both $P_{n}(S)$ and $B_{n}(S)$. Further, since $F_{n-m}(S\setminus \brak{x_{1},\ldots,x_{m}})$ is a finite-dimensional CW-complex and an Eilenberg-Mac~Lane space of type $K(\pi,1)$~\cite{FaN}, the cohomological dimension of $\Gamma_{n-m,m}(S)$ is finite, and the first part follows by taking $m=k$.

%

We now prove the second part, namely that the cohomological dimension of $\Gamma_{n-k,k}(S)$ is equal to $n-k$ for all $n>k$. We first claim that $\cd{\Gamma_{m,l}(S)}\leq m$ for all $m\geq 1$ and $l\geq k-1$. The result holds if $m=1$ since $F_1(S\setminus \brak{x_1,\ldots,x_l})$ has the homotopy type of a bouquet of circles, therefore $H^i(F_1(S\setminus \brak{x_1,\ldots,x_l}), A)$ is trivial for all $i\geq 2$ and for any local coefficients $A$, and $H^1(F_1(S\setminus \brak{x_1,\ldots,x_l}), \Z)\neq 0$. Suppose by induction that the result holds for some $m\geq 1$, and consider the Fadell-Neuwirth short exact sequence:
\begin{equation*}
1\to \Gamma_{1,l+m}(S) \to \Gamma_{m+1,l}(S) \to \Gamma_{m,l}(S) \to 1
\end{equation*}
that emanates from the fibration:
\begin{equation}\label{eq:fngamma}
F_1(S\setminus \brak{x_1,\ldots,x_l,z_{1},\ldots,z_{m}})\to F_{m+1}(S\setminus \brak{x_1,\ldots,x_l}) \to F_{m}(S\setminus \brak{x_1,\ldots,x_l})
\end{equation}
obtained by forgetting the last coordinate. By~\cite[Chapter~VIII]{Br}, it follows that:
\begin{equation*}
\cd{\Gamma_{m+1,l}(S)}\leq \cd{\Gamma_{m,l}(S)}+\cd{\Gamma_{1,l+m}(S)}\leq m+1.
\end{equation*}
which proves the claim. In particular, taking $l=k$, we have $\cd{\Gamma_{m,k}(S)}\leq m$. 

To conclude the proof of the theorem, it remains to show that for each $m\geq 1$ there are local coefficients $A$ such that $H^{m}(\Gamma_{m,l}(S), A)\neq 0$ for all $l\geq k$. We will show that this is the case for $A=\Z$. Again by induction suppose that $H^{m}(\Gamma_{m,l}(S), \Z)\neq 0$ for all $l\geq k-1$ and for some $m\geq 1$ (we saw above that this is true for $m=1$). Consider the Serre spectral sequence with integral coefficients associated to the fibration~\reqref{fngamma}. Then we have that
\begin{equation*}
E^{p,q}_2= H^{p}\bigl(\Gamma_{m,l}(S), H^q(F_1(S\setminus \brak{x_1,\ldots,x_l,z_1,\ldots,z_m}),\Z)\bigr).
\end{equation*}
Since $\cd{\Gamma_{m,l}(S)}\leq m$ and $\cd{F_1(S\setminus \brak{x_1,\ldots,x_l,z_1,\ldots,z_m}}\leq 1$ from above, it follows that this spectral sequence has two horizontal lines whose possible non-vanishing terms occur for $0\leq p \leq m$ and $0\leq q \leq 1$. We claim that the group $E^{m,1}_2$ is non trivial. To see this, first note that $H^1(F_1(S\setminus \brak{x_1,\ldots,x_l,z_1,\ldots,z_m}), \Z)$ is isomorphic to the free Abelian group of rank $r=m+l-k+2$, so $r\geq m+2$, and hence $E^{m,1}_2=H^{m}\bigl(\Gamma_{m,l}(S), \Z^r\bigr)$, where we identify $\Z^r$ with (the dual of) $N\up{Ab}$. The action of $\Gamma_{m,l}(S)$ on $N$ by conjugation induces an action of $\Gamma_{m,l}(S)$ on $N\up{Ab}$. Let $H$ be the subgroup of $N\up{Ab}$ generated by the elements of the form $\alpha(x)-x$, where $\alpha\in \Gamma_{m,l}(S)$, $x\in N\up{Ab}$, and $\alpha(x)$ represents the action of $\alpha$ on $x$. Then we obtain a short exact sequence $0\to H\to N\up{Ab} \to N\up{Ab}/H\to 0$ of Abelian groups, and the long exact sequence in cohomology applied to $\Gamma_{m,l}(S)$ yields:
\begin{equation}\label{eq:lescohom}
\begin{tikzcd}[cramped, column sep=small]
\cdots \arrow{r} & H^m(\Gamma_{m,l}(S), N\up{Ab}) \arrow{r} & H^m(\Gamma_{m,l}(S), N\up{Ab}/H) \arrow{r} & H^{m+1}(\Gamma_{m,l}(S), H)\arrow{r} & \cdots.
\end{tikzcd}
\end{equation}
The last term is zero since $\cd{\Gamma_{m,l}(S)}\leq m$, and so the map between the two remaining terms is surjective. Let us determine $N\up{Ab}/H$. If $S=\St$ then from the comments following \repr{prespinm}, the action of $\Gamma_{m,l}(S)$ on $N\up{Ab}$ is trivial, so $H$ is trivial, and $N\up{Ab}/H\cong \Z^r$. So suppose that $S=\rp$. Choosing the basis
\begin{equation*}
\brak{A_{1,m+l+1},A_{2,m+l+1},\ldots, A_{m+l-1,m+l+1},\rho_{m+l+1}}
\end{equation*}
of $N$ and using \repr{prej}, one sees that the action by conjugation of the generators of $\Gamma_{m,l}(S)$ on the corresponding basis elements of $N\up{Ab}$ is trivial, with the exception of that of $\rho_{i}$ on $A_{i,m+l+1}$ for $l+1\leq i\leq m+l-1$, which yields elements $A_{i,m+l+1}^2\in H$ (by abuse of notation, we denote the elements of $N\up{Ab}$ in the same way as those of $N$), and that of $\rho_{i}$ on $\rho_{m+l+1}$, where $l+1\leq i\leq m+l$, which yields elements $A_{i,m+l+1}\in H$. In the quotient $N\up{Ab}/H$ the basis elements $A_{l+1,m+l+1},\ldots, A_{m+l-1,m+l+1}$ thus become zero, and additionally, we have also that $A_{m+l,m+l+1}$ (which is not in the given basis) becomes zero. Hence the  relation $\prod_{i=1}^{m+l}\, A_{i,m+l+1}=\rho_{m+l+1}^{-2}$ is sent to the relation $\prod_{i=1}^{l}\, A_{i,m+l+1}=\rho_{m+l+1}^{-2}$, and so $N\up{Ab}/H$ is generated by (the images of) the elements $A_{1,m+l+1}, \ldots, A_{l,m+l+1}, \rho_{m+l+1}$, subject to this relation (as well as the fact that the elements commute pairwise). It thus follows that $N\up{Ab}/H\cong \Z^l$. Since the induced action of $\Gamma_{m,l}(S)$ on $N\up{Ab}/H$ is trivial, we conclude that 
\begin{equation*}
H^m(\Gamma_{m,l}(S), N\up{Ab}/H)=\bigl( H^m(\Gamma_{m,l}(S), \Z)\bigr)^{s},
\end{equation*}
where $s=m+l$ if $S=\St$ and $s=l$ if $S=\rp$. It then follows from \req{lescohom} that $E^{m,1}_2=H^m(\Gamma_{m,l}(S), N\up{Ab})\neq 0$. Since  $E^{p,q}_2=0$ for all $p>m$ and $q>1$, we have $E^{m,1}_2=E^{m,1}_{\infty}$, thus $E^{m,1}_{\infty}$ is non trivial, and hence $H^{m+1}(\Gamma_{m+1,l}(S), \Z)\neq 0$. This concludes the proof of the theorem.
\end{proof}

We end this paper with a proof of \reco{vcdmcg}.

\begin{proof}[Proof of \reco{vcdmcg}]
Let $S=\St$ (resp.\ $S=\rp$). If $n\geq 3$ (resp.\ $n\geq 2$) then $B_{n}(S)$ and $\mathcal{MCG}(S,n)$ are closely related by the following short exact sequence~\cite{Sc}:
\begin{equation*}
1 \to \bigl\langle \ft\bigr\rangle\to B_{n}(S)\stackrel{\beta}{\to} \mathcal{MCG}(S,n) \to 1,
\end{equation*}
where the kernel is isomorphic to $\Z_{2}$. Now assume that $n\geq 4$ (resp.\ $n\geq 3$), so that $B_{n}(S)$ is infinite. If $\Gamma$ is a torsion-free subgroup of $B_{n}(S)$ of finite index then $\beta(\Gamma)$, which is isomorphic to $\Gamma$, is a torsion-free subgroup of $\mathcal{MCG}(S,n)$ of finite index, and hence the virtual cohomological dimension of $\mathcal{MCG}(S,n)$ is equal to that of $B_{n}(S)$. The result then follows by \reth{prop12}.
\end{proof}

\end{document}